\documentclass[aos,noinfoline]{imsart}

\makeatletter
\def\journal@name{$ $} 
\makeatother

\usepackage[utf8]{inputenc}

\usepackage{amssymb,amsmath,latexsym,amsthm}
\usepackage{graphicx}
\usepackage{xcolor, colortbl}
\usepackage{tikz}
\usetikzlibrary{arrows,calc,positioning,shapes}
\usetikzlibrary{external}
\usepackage{booktabs}
\usepackage{siunitx}
\usepackage{dcolumn}

\usepackage{url}
\usepackage{verbatim}
\usepackage{algorithm}
\usepackage{algorithmic}

\usepackage[caption=false]{subfig}
\usepackage{enumitem}

\usepackage[round]{natbib}

\usepackage[colorlinks]{hyperref}

\newcolumntype{d}{D{.}{.}{-1}}

\newcommand{\face}{F}

\DeclareMathOperator{\rank}{rank}
\DeclareMathOperator{\supp}{supp}

\DeclareMathOperator{\MLE}{MLE}
\DeclareMathOperator{\EMLE}{EMLE}

\newcommand{\hmuFF}{\hat{\mu}^{F\smash{'_2\setminus F'_1}}}

\definecolor{LRed}{rgb}{1,.8,.8}
\def\R{{\mathbf R}}
\def\F{{\mathbb F}}

\def\N{{\mathbb N}}

\def\<{\langle}
\def\>{\rangle}

\def\i{\underline i}

\def\t{\theta}
\def\tl{\triangleleft}

\newcommand{\Ecal}{\mathcal{E}}

\newcommand{\Fbf}{\mathbf{F}}
\newcommand{\Pbf}{\mathbf{P}}

\newcommand{\gmle}{p_{*}}

\newcommand{\ol}[1]{\overline{#1}}

\newcommand{\be}{\begin{equation}}
\newcommand{\ee}{\end{equation}}

\newtheorem{theorem}{Theorem}[section]

\newtheorem{lemma}[theorem]{Lemma}
\theoremstyle{remark}
\newtheorem{remark}[theorem]{Remark} 
\theoremstyle{definition}

\newtheorem{example}[theorem]{Example} 

\newcommand{\griddefinenodes}[2]{
  \foreach \j in {1,...,#2} \foreach \i in {1,...,#1} {
    \path (\j,-\i) coordinate (X\i\j);}}
\newcommand{\griddrawedges}[2]{
  \foreach \i in {1,...,#1} \foreach \j [count=\jj] in {2,...,#2} { \draw (X\i\jj) -- (X\i\j);}
  \foreach \j in {1,...,#2} \foreach \i [count=\ii] in {2,...,#1} { \draw (X\ii\j) -- (X\i\j);}}
\newcommand{\griddrawnodes}[2]{
  \foreach \j in {1,...,#2} \foreach \i in {1,...,#1} {
    \filldraw[fill=green!50!black] (X\i\j) circle (3pt); }}
\newcommand{\gridaddlabels}[3][1]{
  \foreach \j in {1,...,#3} {
    \node[label=above:{\pgfmathparse{#2*(\j-1)+#1}\pgfmathprintnumber{\pgfmathresult}}] at (X1\j) {};
    \node[label=below:{\pgfmathparse{#2*\j-1+#1}\pgfmathprintnumber{\pgfmathresult}}] at (X#2\j) {}; }}
\newcommand{\grid}[2]{
  \griddefinenodes{#1}{#2}
  \griddrawedges{#1}{#2}
  \griddrawnodes{#1}{#2}}
\newcommand{\labelledgrid}[3][1]{
  \grid{#2}{#3}
  \gridaddlabels[#1]{#2}{#3}}

\newcommand\overmat[2]{\smash{
  \overbrace{#2\strut}^{\text{\color{black}#1}}}}

\usepackage{changes}
\definechangesauthor[color=green]{jr}

\begin{document}
\begin{frontmatter}
\title{Approximating faces of marginal polytopes in discrete hierarchical models}
\runtitle{Approximating faces of marginal polytopes}

\begin{aug}
  \author{\fnms{Nanwei} \snm{Wang}\corref{}\ead[label=nw]{wangnanw@yorku.ca}}
  \address{\printead{nw}}
\affiliation{York University, Toronto, Canada}
\and
\author{\fnms{Johannes} \snm{}Rauh\ead[label=jr]{jrauh@mis.mpg.de}}
\address{\printead{jr}}
\affiliation{York University, Toronto, Canada}
\affiliation{Max Planck Institute for Mathematics in the Sciences, Germany}
\and
\author{\fnms{H\'el\`ene} \snm{Massam}\ead[label=hm]{massamh@yorku.ca}}
\address{\printead{hm}}
\affiliation{York University, Toronto, Canada}
\thankstext{t1}{H. Massam gratefully acknowledges support from an NSERC Discovery Grant}
\runauthor{ Wang \&  Rauh \&  Massam}
\end{aug}
\begin{abstract}
  The existence of the maximum likelihood estimate in hierarchical loglinear models is crucial to the reliability of
  inference for this model. Determining whether the estimate exists is equivalent to finding whether the sufficient
  statistics vector $t$ belongs to the boundary of the marginal polytope of the model.  The dimension of the smallest
  face $\Fbf_t$ containing $t$ determines the dimension of the reduced model which should be considered for correct
  inference.  For higher-dimensional problems, it is not possible to 
 compute  $\Fbf_{t}$ exactly.
  \citet{MassamWang15:local_approach} found an outer approximation to $\Fbf_t$ using a collection of sub-models of the
  original model.  This paper refines the methodology to find an outer approximation and devises a new methodology to
  find an inner approximation.  The inner approximation is given not in terms of a face of the marginal polytope, but in
  terms of a subset of the vertices of~$\Fbf_t$.

  Knowing $\Fbf_t$ exactly indicates which cell probabilities have maximum likelihood estimates equal to $0$.
   When  $\Fbf_t$ cannot be obtained exactly, we can use, first, the outer approximation $\Fbf_2$ to  reduce the dimension of the problem and, then, the inner approximation $\Fbf_1$ to obtain correct estimates of cell probabilities corresponding to elements of $\Fbf_1$ and improve the estimates of the remaining probabilities corresponding to elements in $\Fbf_2\setminus \Fbf_1$.
Using both real-world and simulated data, we illustrate our results, and show  that our methodology scales to high dimensions. 

\vspace{.5cm}

\noindent
\textbf{Keywords:} existence of the maximum likelihood estimate, marginal polytope, faces, facial sets, extended maximum likelihood estimate.
\vspace{.5cm}
 
 \end{abstract}



\end{frontmatter}



\section{Introduction}
Discrete hierarchical models are an essential tool for the  analysis of categorical data given under the form of a contingency table.  
The study of these models goes back more than a century, and a detailed history of  their development is given in \citet{FienbergRinaldo07:three_centuries}. Nowadays, discrete hierarchical models are used for the analysis of large sparse contingency tables where many, if not most, of the entries are small or zero counts. It is well-known that in such cases, the maximum likelihood estimate (henceforth abbreviated MLE)  of the parameters may not exist. The non existence of the MLE has problematic consequences for inference, clearly for estimation, but also for testing and model selection.  \citet{FienbergRinaldo12:MLE_in_loglinear_models}  list the statistical implications of the non existence of the MLE, such as the unreliability of the estimates of some of the parameters or the usage of the wrong degrees of freedom for testing one model against another.
\citet{Geyer09:Likelihood_inference_in_exponential_families} describes the problems attached to the nonexistence of the MLE and presents an R program that yields meaningful confidence intervals and tests.  \citet{LetacMassam12:Bayes_regularization} study the statistical implications of the nonexistence of the MLE on model selection in Bayesian inference.

\citet{FienbergRinaldo12:MLE_in_loglinear_models} also give necessary and sufficient conditions for the existence of the MLE, which
 are not restricted to hierarchical models, but  apply to all discrete exponential families (loglinear models).
These conditions are extensions of results given earlier by
\citet{Haberman74:Analysis_of_Frequency_Data}, \citet{Barndorff:exponential_families} and
\citet{EFRS06:Polyhedral_conditions_MLE}.
They are essentially as follows.
Denote by $I$ the set of outcomes of a statistical experiment, that is the set of cells of the contingency table where the data is classified.
Let ${\cal I}_{+}=\{i_{1},\dots,i_{N}\}$  be the outcome of $N$ independent repetitions of the experiment within either a multinomial or Poisson setting.  The data ${\cal I}_{+}$ is summarized by the vector $t$ of sufficient statistics, which is of the form $t=\sum_{j=1}^Nf_{i_j}$  for  some vectors $f_{i}$, $i\in I$, determined by the given hierarchical loglinear model.
Under these assumptions, the distribution of the data belongs to a natural exponential family with density
\begin{equation}
  \label{basic}
  f(i_{1},\dots,i_{N}; \theta) =\exp \{\langle \theta, t\rangle-Nk(\theta)\}
\end{equation}
with respect to the counting measure, where $\theta$ is a loglinear parameter.
To each exponential family is associated a
polytope~$\Pbf$, called the marginal polytope, which is the convex hull of the vectors $f_{i}, i\in I$.
Furthermore, $\Pbf$ contains all possible realizations of $\frac t N$ for arbitrary repetitions
of the statistical experiment.  For given data and a given hierarchical model, the MLE then exists if and only if
$\frac tN$ belongs to the relative interior of~$\Pbf$.  If the MLE does not exist, then $\frac tN$ belongs to the
relative interior of a face denoted $\Fbf_{t}$. It is the smallest face of~$\Pbf$ containing $\frac tN$, and it is
proper (i.e.~$\Fbf_{t}\neq\Pbf$).  Thus, determining whether, for a given data set, the MLE of the parameter of a discrete hierarchical
loglinear model exists is equivalent to determining whether $\frac t N$ belongs to a proper face of~$\Pbf$. The parameter  may be  the loglinear parameter $\theta$, or  the cell probabilities $p=(p(i), i\in I)$ obeying the constraints of the model and in 1-1 correspondence with $\theta$. The MLE can thus be thought
of in terms of $\theta$, or in terms of $p$.

If the MLE does not exist, it is still possible to compute the extended MLE
(EMLE)~\citep{Barndorff:exponential_families,Lauritzen:Graphical_models,CsiszarMatus08:EMLE}, which is a probability
distribution that maximizes the likelihood over the closure of the hierarchical model (that is, the EMLE can be
approximated arbitrarily well by distributions from the hierarchical model).
The support of the EMLE is
given by the set $F_t=\{i\in I: f_{i}\in\Fbf_{t}\}$, called the facial set of $\Fbf_t$.  When this support is known, computing the EMLE is equivalent to an ordinary
MLE computation on a smaller exponential family~$\Ecal_{F_{t}}$, of dimension~$\dim(\Fbf_{t})$,
 generated by a measure with support~$F_t$ \citep{Geyer09:Likelihood_inference_in_exponential_families}. 
Therefore, precise knowledge of~$\Fbf_{t}$ and $F_t$ yields which is the proper dimension of the model to be used in testing and which outcomes $i\in I$ are attributed a probability of $0$ by the EMLE,
and it allows us to compute the EMLE. One should also note that the usual regularity conditions used for the asymptotic properties of the MLE, which are not satisfied for the given model when the MLE does not exist, are satisfied for the reduced model~$\Ecal_{F_{t}}$.

The problem is then to find~$\Fbf_{t}$.  
This is easy when the face lattice of $\Pbf$ is known or can be computed using a standard discrete geometry toolkit such as, for example, \texttt{polymake}~\citep{polymake}.
For some classes of marginal polytopes, the face lattice is known, for example for decomposable models and no-three-way-interaction models with small variables \citep{Vlach86:3D_planar_transportation_problem}.
For binary variables, the marginal polytope is a cut polytope \citep{DezaLaurent09:Cuts_and_metrics}.
Other authors have studied convex support polytopes, which replace marginal polytopes for more general exponential families.
Notably, many such polytopes have been described for exponential random graph models, see, for example, \cite{KarwaSlavkovic16:MLE_in_betamodels} and papers cited therein.
When the face lattice of $\Pbf$  cannot be computed, algorithms to compute $\Fbf_{t}$ that are based on linear programming have been proposed by \citet{EFRS06:Polyhedral_conditions_MLE}, by \citet{Geyer09:Likelihood_inference_in_exponential_families},
and by \citet{FienbergRinaldo12:MLE_in_loglinear_models}. 
These methods, however, become computationally infeasible in large dimensions, which happens, in our experience, for
hierarchical models when the set of random variables $V$ contains more than 16 binary variables (or correspondingly
fewer larger variables).

For larger models, \citet{MassamWang15:local_approach} propose to approximate $\Fbf_t$ by relating it to faces of
smaller hierarchical models as follows.  A hierarchical model for the discrete random variable $X=(X_v, v\in V)$ is
determined by a set of interactions among its components $X_v, v\in V,$ that is represented by a simplicial
complex~$\Delta$.  \citet{MassamWang15:local_approach} consider subsets $V_i,i=1,\ldots,k,$ of $V$ containing less than
16 variables and the hierarchical models Markov with respect to the induced simplicial
subcomplexes. 
Linear programming can be used to compute the smallest faces $\Fbf_{t_i}$ containing the corresponding sufficient
statistic $t_i, i=1,\ldots,k$.  These faces, which a priori are faces of the marginal polytopes of the submodels,
naturally correspond to faces of the original marginal polytope.  \citet{MassamWang15:local_approach} prove that the
intersection of these 
is a face $\Fbf_{2}$ of $\Pbf$ containing~$\Fbf_{t}$.  Thus, if $\Fbf_{2}$ is a proper face of~$\Pbf$, then $\Fbf_{t}$
is necessarily a proper face, and therefore the MLE does not exist.  While \citet{MassamWang15:local_approach} work with
graphical models, we show that their results are also true for hierarchical loglinear
models.

We call $\Fbf_{2}$ an \emph{outer approximation} to~$\Fbf_{t}$.
This is similar to the notion of an outer approximation in optimization, which describes a polytope that
  contains the original polytope of interest.  While the outer approximation polytope in optimization usually has the
  same dimension as the original polytope, the outer approximation face $\Fbf_{2}$ does not necessarily have the same
  dimension as~$\Fbf_{t}$.

The purpose of this paper is to add to this outer approximation ${\mathbf F}_2$ an inner approximation~${\mathbf F}_1$ that is a subset of ${\mathbf F}_t$.
While $\Fbf_{2}$ is derived from looking at simplicial subcomplexes of~$\Delta$, 
the inner approximation is constructed by enlarging the simplicial complex through added interactions.  In particular, we propose a process of ``completing a separator,'' which leads to a decomposable simplicial complex which, in turn, can be studied by looking at the sub-simplices  corresponding to its components with a small number of vertices in~$V$.  Thus, both $\Fbf_{1}$ and~$\Fbf_{2}$ can be obtained by computing facial sets on smaller hierarchical models involving fewer nodes.

The inner and outer approximations to $\Fbf_{t}$ satisfy
$${\mathbf F}_1\subset {\mathbf F}_t\subset {\mathbf F}_2.$$
Clearly, we want  ${\mathbf F}_1$ as large as possible and ${\mathbf F}_2$ as small as possible to have as much information about ${\mathbf F}_t$ as possible.
In our simulations, we observe that $\Fbf_{t}=\Fbf_{2}$ most of the time and that $\Fbf_{t}=\Fbf_{1}$ quite often.
The approximations $\Fbf_{1}$ and~$\Fbf_{2}$ allow to bound the dimension of~$\Fbf_{t}$, and thus knowledge from
$\Fbf_{1}$ and~$\Fbf_{2}$ can be taken into account whenever the dimension of $\Fbf_{t}$ plays a role, for example in
hypothesis testing.

When the MLE does not exist, even though the maximum likelihood procedure cannot be used to obtain a point estimate for the
parameter vector~$\theta$, some of its components $\theta_{j}$ may still be finite and well-defined in this
situation.  In Section~\ref{sec:parameter-estimation} we introduce a loglinear parametrization~$\mu$, different from
$\theta$, that allows to say precisely which parameter combinations have a finite well-defined limit and thus remain
meaningful for statistical inference.  Moreover, we demonstrate that even when $\Fbf_{t}$ is unknown, the
parametrization $\mu$ can be adjusted to incorporate knowledge that is available in the form of inner and outer
approximations $\Fbf_{1}$ and~$\Fbf_{2}$.

We extend the work of \citet{FienbergRinaldo12:MLE_in_loglinear_models} and that of \cite{Geyer09:Likelihood_inference_in_exponential_families} in several directions: first, we construct approximations to $\Fbf_{t}$ in high dimensions when a direct computation of $\Fbf_{t}$ is not feasible.  Second, we explicitly identify all parameter combinations that remain finite and meaningful when the MLE does not exist and $\Fbf_{t}$ is known, and we also discuss what can be said  when only approximations to $\Fbf_{t}$ are available.

\medskip

The remainder of this paper is organized as follows. In Section~\ref{sec:preliminaries}, we give preliminaries on hierarchical models, and faces and facial sets. Section~\ref{sec:facial-sets} contains the original methodology to obtain the approximations
$\Fbf_1$ and $\Fbf_2$.
In Section~\ref{sec:parameter-estimation}, we show how to use $\Fbf_1$ and $\Fbf_2$ to identify the parameters of the hierarchical models that can be estimated and those that cannot. 
In Section~\ref{sec:experiments}, we present  two examples.  A simulated data set is used to assess how often our approximations succeed to identify the true facial set~$F_{t}$.  The NLTCS data set, studied by \citet{DobraEroshevaFienberg03} and \citet{DobraLenkoski11:Copula_Gaussian_graphical_models}, illustrates how the outer approximation $\Fbf_2$ improves estimates of cell probabilities and log-linear parameters.
Both of these examples have 16 nodes. In Section~\ref{sec:large-graphs}, we discuss how to apply the methodology to larger models and how to use for inference the information that it yields.  Two examples illustrate this:
simulated data from the graphical model of the $5\times 10$ grid, and the real-world data set of voting records in the US Senate.

  Appendix~\ref{S-sec:notation-hier} describes the concrete parametrization that we use in the examples.
  Appendix~\ref{S-sec:two-binaries} discusses the case of two binary variables to illustrate what happens to the usual parameters when the MLE does not exist.
  Appendix~\ref{S-sec:best-parameters} discusses how to further improve the parametrization $\mu_{L}$ introduced in Section~\ref{sec:mus-0}.
  Appendices~\ref{S-sec:uniform-4x4} and \ref{S-sec:NLTCS-frequencies} give further results for the examples from Section~\ref{sec:experiments}.
  Appendix~\ref{S-linear programming} briefly summarizes the linear programming algorithm to compute~$\Fbf_{t}$ by \citet{FienbergRinaldo12:MLE_in_loglinear_models}.
  Appendix~\ref{sec:S-proofs} contains proofs of the results of \citet{Barndorff:exponential_families} that describe the closure of an exponential family and the EMLE.
  
Our results apply not only to hierarchical models, but to arbitrary discrete exponential families.  In this paper, the focus is on
hierarchical and graphical models, which are the main application,  and for which the construction
of the inner and outer approximations can be described in terms of the underlying simplicial complex or graph.

\section{Preliminaries}
\label{sec:preliminaries}
\label{sec:mus-0}

In the following four subsections, 
we recall basic facts about hierarchical models, discrete exponential families, polytopes and the closure of
exponential families, and we define the extended MLE.

\subsection{Hierarchical models and discrete exponential families}
\label{sec:hierarchical-models}
\label{sec:discr-expon-famil}

For details and proofs on the material in this subsection, we refer to \citet{LetacMassam12:Bayes_regularization}
and~\citet{RKA10:Support_Sets_and_Or_Mat}.  Let $X=(X_v,v\in V)$ be a discrete random vector with components indexed by
a finite set $V=\{1,\ldots,p\}$.  Each variable $X_v$ takes values in a finite set $I_v, v\in V$.  The vector $X$ takes
its values in
$
I=\prod_{v\in V}I_v$, 
the set of cells $i=(i_v, v\in V)$ of a $p$-dimensional contingency table.
For any $D\subseteq V$, the subvector $X_{D}=(X_{v})_{v\in D}$ takes its values in $I_{D}=\prod_{v\in D}I_{v}$.
The $D$-marginal cell of $i\in I$ will be denoted by $i_{D}=(i_{v})_{v\in D}$.
The corresponding restriction is the coordinate projection map $i\mapsto i_{D}$ and is denoted by~$\pi_{D}$.

Let ${\Delta}$ be a simplicial complex on $V$, that is, $\Delta$ is a set of subsets $D\subset V$  such that $D\in \Delta$ and $D'\subset D$ implies $D'\in \Delta$.
The joint distribution of $X$ is \emph{hierarchical} with underlying simplicial complex ${\Delta}$ (or generating set $\Delta$)
if the probability $p(i)=P(X=i)$ of a single cell $i=(i_v, v\in V)$ is of the form
\begin{equation}
  \label{H2prime-redundant}
  \log p(i)=\sum_{D\in \Delta} \theta_D(i_D)
\end{equation}
where $\theta_D(i_D)$ is a function of the marginal cell $i_D=(i_v, v\in D)$ only.  To make precise the dependence on~$\theta$, we sometimes write~$p_{\theta}(i)$ instead of~$p(i)$.  The set of all such distributions $\Ecal_{\Delta}:=\{p_{\theta}\}$ is called the \emph{hierarchical model} of~$\Delta$.

Equation~\eqref{H2prime-redundant} is  essentially a linear condition on $\log p(i)$.  It is possible to parametrize the hierarchical model using a finite vector of parameters $(\theta_{j})_{j\in J}$ such that
\begin{equation}
  \label{eq:log-linear}
  \log p_{\theta}(i) = \sum_{j\in J}\theta_{j} a_{j,i} - k(\theta),
\end{equation}
where $A_{\Delta}=(a_{j,i})_{j\in J,i\in I}$ is a fixed real matrix (depending only on $\Delta$) and where
$k(\theta) = \log \sum_{i}\exp(\sum_{i}\theta_{j}a_{j,i})$ ensures the normalization
$\sum_{i\in I}p_{\theta}(i) = 1$.  This parametrization is not unique.  In the examples, we use an explicit
parametrization that is used, for example, by~\cite{LetacMassam12:Bayes_regularization}.  For convenience, we recall
this parametrization in Appendix~\ref{S-sec:notation-hier}.

An important subclass of hierarchical models is the class of graphical models. Let $G=(V,E)$ be an undirected graph with
vertex set $V$ and edge set~$E$.  A subset $D\subseteq V$ is a clique of $G$ if for any $i,j\in V$, $i\neq j$, the
edge $(i,j)$ is in $E$.  The set of cliques of~$G$, denoted by~$\Delta(G)$, is a simplicial complex.  The \emph{graphical
  model} of $G$ is defined as the hierarchical model of~$\Delta(G)$.  Graphical models are important because of their
interpretation in terms of conditional independence, see~\cite{Lauritzen:Graphical_models}.



\medskip

Hierarchical models are examples of discrete exponential families, see
\citet{Barndorff:exponential_families,Fienberg:categorical,RKA10:Support_Sets_and_Or_Mat}.
Generalizing~\eqref{eq:log-linear}, let $I$ and $J$ be finite sets and let $A\in\R^{J\times I}$ be a real
matrix.  
Denote the columns of $A$ by $f_{i}$, $i\in I$.  The discrete exponential family corresponding to~$A$, denoted
by~$\Ecal_{A}$, consists of all probability distributions on~$I$ that are of the form
\begin{equation}
\label{pofi}
  p_{\theta}(i) = \exp\big\{\langle \theta,f_{i}\rangle - k(\theta)\big\}
  = \exp\big\{(A^{t}\theta)_{i} - k(\theta)\big\}, \qquad \theta\in\R^{J},
\end{equation}
where, as above, $k(\theta) = \log \sum_{i}\exp(\sum_{j}\theta_{j}a_{j,i})$.
It is convenient to write $\tilde A$ for the $(1+|J|)\times I$ matrix with columns equal to $\binom{1}{f_{i}}$, $i\in I$, and to set $\theta_{0}:= -k(\theta)$ and $\tilde{\theta}=(\theta_{0}, \theta)$ (as a column vector).  Then~\eqref{pofi} rewrites to
\begin{equation}
  \label{eq:pofi2}
  p_{\theta}(i) = \exp\big(\tilde A^{t}\tilde\theta\big), \qquad \theta\in\R^{J}.
\end{equation}
Both $A$ and $\tilde A$ are called \emph{design matrices} of the model.
The convex hull of the columns $f_{i}$, $i\in I$, is called
the \emph{convex support polytope}, denoted by~$\Pbf_{A}$.  
In the case of a graphical or hierarchical model, $\Pbf_{\Delta}:=\Pbf_{A_{\Delta}}$ is called a \emph{marginal polytope}.

The parametrization $\theta\to p_{\theta}$ is identifiable if and only if $\tilde A$ has full rank.
If $\tilde A$ does not have full rank, then one can drop rows of~$A$ to obtain a submatrix~$A'$ such
that~$\tilde A'$ has full rank.  This is equivalent to setting certain parameters to zero until the remaining parameters
are identifiable. 

Later, the following reparametrization will be useful: select an element of~$I$, which we will denote by $0$.  Let
$A_{0}$ be the matrix with columns $f_{i}-f_{0}$, $i\in I\setminus\{0\}$.  It is not difficult to see that $A$ and
$A_{0}$ define the same exponential family (since $\tilde A$ and $\tilde A_{0}$ have the same row span).  Let
$h'=\rank(A_{0})=\rank(\tilde A_{0})-1$, and select a set $L$ of $h'$ linearly independent vectors among the columns
of~$A_{0}$.
For $i\in L$, let $\mu_{i}=\mu_{i}(\theta):=\langle\theta,f_{i}-f_{0}\rangle$, and let $\mu_{L}=(\mu_{i},i\in L)$.
Then the $\mu_{L}$ are identifiable parameters on~$\Ecal_{A}$: in fact, their number is equal
to~$h'$, and they are independent by construction.

It is possible to extend the definition of $\mu_{i}(\theta)$ to all $i\in I$.  Note that only the parameters $\mu_{i}$
with $i\in L$ are free parameters, while the parameters $\mu_{i}$ with $i\in I\setminus L$ are linear functions
of~$\mu_{L}$.
The $\mu_{i}$ can be interpreted as log-likelihood ratios:
\begin{equation*}
  \mu_{i}(\theta)=\log \frac{p_{\theta}(i)}{p_{\theta}(0)},
  \qquad
  \mu_{0}(\theta)=0.
\end{equation*}


Let $n=(n(i),i\in I)$ be an $I$-di\-men\-sional column vector of cell counts summarizing the outcome of a statistical experiment.
Then
\begin{equation}
  \label{marginal}
  \tilde{A}n=\left(\begin{array}{c}N\\t\end{array}\right)
  \quad\text{ and }\quad
  A n= t,
\end{equation}
where $N=\sum_{i\in I}n(i)$ is the total cell counts and $t$ is the column vector of \emph{sufficient statistic}.  The
likelihood $\prod_{i\in I}p_{\theta}(i)^{n(i)}$ can be written under the form of a natural exponential family.  Indeed,
\begin{equation*}
  \prod_{i\in I}p_{\theta}(i)^{n(i)}
  = \exp\big( \big\langle \tilde{A}n, \tilde{\theta}\big\rangle\big) = \exp\big\{\sum_{j\in J}\theta_jt_j - N k(\theta)\big\}.
\end{equation*}
The log-likelihood function for the loglinear parameters $\theta$ of $\Ecal_{A}$ is therefore
\begin{equation}
  \label{lik-theta}
  l(\theta)=
  \sum_{j\in J}\theta_jt_j-N k(\theta) .
\end{equation}
It is well-known that $l(\theta)$ is concave.  If the parameters are identifiable, then it is strictly concave.
We can also express the log-likelihood as a function 
of $\mu=(\mu_i, i\in I)$:
\begin{equation}
  l(\mu ) = \sum_{i\in I}n(i)\log p(i) = \sum_{i\in I}n(i)\mu_i-N\log (\sum_{i\in I}\exp \mu_i)    \label{lik-mu}.
\end{equation}
As stated before, only a subset $\mu_{L}$ of the parameters $\mu$ are independent, and the remaining $\mu_{i}$, $i\notin
L$, can be expressed as linear functions of~$\mu_{L}$.

\subsection{The convex support and its facial sets}
\label{sec:polytopes}

We next recall some facts about facial sets.  We refer to \citet{Ziegler98:Lectures_on_Polytopes} for a general
introduction to polytopes and their face lattices.




The convex support polytope $\Pbf_{A}$ is defined as the convex hull of a finite number of points~$f_{i}$, $i\in I$.  It
is of interest to know which subsets of $\{f_{i}\}_{i\in I}$ lie on a given face~$\Fbf$.
Thus, we describe a face $\Fbf$ by identifying the corresponding \emph{facial set} $F = \{i\in I: f_{i}\in\Fbf\}$.  For
any subset~$S\subseteq I$, denote by~$\face_{A}(S)$ the smallest facial set that contains~$S$.  The intersection of
facial sets is again facial, and so $\face_{A}(S)$ is well-defined.  When $\Pbf_{A}=\Pbf_{\Delta}$ is a marginal
polytope, we abbreviate $\face_{A_\Delta}(S)$ by~$\face_\Delta(S)$.

As mentioned in the introduction, to derive the inner approximation $\Fbf_1$ to $\Fbf_t$ and its outer approximation
$\Fbf_2$, we need to consider sub-models of a given model.
When one exponential family $\Ecal_{A'}$ is a subset of another family~$\Ecal_{A}$, then the convex support polytope
$\Pbf_{A'}$ is a linear projection of~$\Pbf_{A}$, and the columns $f'_{i}$ of $A'$ are indexed by the same set $I$ as
the columns $f_{i}$ of~$A$.  Since inverses of linear projections preserve faces, it follows from basic results about
polytopes that $\face_{A}(S)\subseteq\face_{A'}(S)$; see Chapter~1 in~\citet{Ziegler98:Lectures_on_Polytopes}.  For
hierarchical models, these facts are summarized in the following result:

\begin{lemma}
  \label{lem:sub-complex}
  Let $\Delta$ and $\Delta'$ be simplicial complexes on the same vertex set with $\Delta'\subseteq\Delta$.
  Then $\Pbf_{\Delta'}$ is a coordinate projection of~$\Pbf_{\Delta}$.  The inverse image of any face of $\Pbf'$ is a face of $\Pbf$. Moreover, for any $S\subseteq I$, we have
  $\face_\Delta(S)\subseteq\face_{\Delta'}(S)$.
\end{lemma}

\begin{remark}
  \label{homog}
  It is convenient to embed $\Pbf_{A}$ in a vector space with one additional dimension using a map
  $\R^{h}\to\R^{h+1}, t\mapsto\tilde t:=(1,t)$.  This has the advantage that all defining inequalities are brought
  into a homogeneous form with vanishing constant: note that
  $\langle g, f_{i}\rangle - c = \langle \tilde g_{c}, \tilde f_{i}\rangle$, where $\tilde g_{c} := (-c,g)$.

  When a defining inequality of a face $\Fbf$ is given, its facial set~$F$ can be obtained by checking whether
  $f_{i}\in\Fbf$ for each~$i\in I$.  In the other direction, when a facial set $F$ is given, it is much more difficult
  to compute a defining inequality of the corresponding face~$\Fbf$.  However, it is straightforward to compute the
  linear equations defining~$\Fbf$: the set of such equations $0 = \<g,x\> - c = \<\tilde g,\tilde x\>$ corresponds to
  the set of vectors~$\tilde g\in\ker\tilde A_{F}^{t}$, where $\tilde A_{F}$ is the matrix obtained from~$\tilde A$ by
  dropping the columns not in~$F$.
\end{remark}

\subsection{The closure of an exponential family and existence of the MLE}

For a family $\Ecal_{A}$ and cell counts $n=(n(i): i\in I)$ given as above, a parameter value $\theta^{*}$ is an MLE if it is a global maximum of~$l(\theta)$.
An MLE need not exist, since the domain of the parameters $\theta$ is unbounded.  The likelihood can also be written as a function of cell probabilities.  For any probability
distribution $p$ on~$I$ let
\begin{equation*}
  \tilde l(p) = \log\{\prod\nolimits_{i\in I}p(i)^{n(i)}\}.
\end{equation*}
Then $l(\theta) = \tilde l(p_{\theta})$, and $\theta^{*}$ is an MLE if and only if $p_{\theta^{*}}$ maximizes $\tilde l$
subject to the constraint that $p$ belongs to~$\Ecal_{A}$, i.e. is of the form~\eqref{pofi}.
When $\tilde l$ has no maximum on~$\Ecal_{A}$, we can pass to the topological closure~$\overline{\Ecal_{A}}$.
It can be characterized in terms of the convex support polytope $\Pbf_{A}$ and its facial sets
as follows:
\begin{theorem}[\citet{Barndorff:exponential_families}]
  \label{thm:closure-of-expfam}
   The topological closure of $\Ecal_{A}$ is $\overline{\Ecal_{A}} = \bigcup_{F}\Ecal_{F,A}$,
   where $F$ runs over all facial sets of 
   $\Pbf_{A}$ and where $\Ecal_{F,A}$ consists of all probability distributions of the form
  $p_{F,\theta}$, with
  \begin{equation}
    \label{eq:p-f-theta}
    p_{F,\theta}(i) =
    \begin{cases}
      \exp(\langle \theta, f_i\rangle  - k_{F}(\theta)), & \text{ if }i\in F, \\
      0, & \text{ otherwise},
    \end{cases}
  \end{equation}
  where $k_{F}(\theta) = \log \sum_{i\in F}\exp(\langle \theta, f_i\rangle)$.
\end{theorem}
\begin{proof}
  See~\cite{Barndorff:exponential_families}.
  For self-containedness we provide a proof in our notation in Appendix~\ref{S-sec:proof1}.
\end{proof}

Thus, $\ol{\Ecal_{A}}$ is a finite union of sets $\Ecal_{F,A}$ that are exponential families themselves with a very
similar parametrization, using the same number of parameters.  The design matrix of $\Ecal_{F,A}$ is the submatrix
$A_{F}$ of $A$ consisting of the columns indexed by~$F$.  However, for any proper facial set~$F\neq I$, the
parametrization $\theta\mapsto p_{F,\theta}$ is never identifiable since all
columns of $A_{F}$ lie on a supporting hyperplane defining~$F$ and thus $\tilde A_{F}$ never has full rank.

Although the parameters $\theta$ on $\Ecal_{A}$ and the parameters $\theta$ on
$\Ecal_{F,A}$ play similar roles, they are very different in the following sense: if $\theta^{(s)}$ is a sequence of
parameters with $p_{\theta^{(s)}}\to p_{F,\theta}$ for some~$\theta$, then, in general,
$\lim_{s\to\infty}\theta^{(s)}_{j}\neq\theta_{j}$ for all~$j\in J$.
\begin{theorem}[\cite{Barndorff:exponential_families}]
  \label{thm:gmle}
  For any vector of observed counts~$n$, there is a unique maximum $p^{*}$ of~$\tilde l$ in~$\overline{\Ecal_{A}}$.
  This maximum $p^{*}$ satisfies:
  (1) $A p^{*} = \frac tN$, where $t = An$,
  (2) $\supp(p^{*}) = F_{t}$,
  (3) $p^{*}\in\Ecal_{F_{t},A}$.
\end{theorem}
\begin{proof}
  See~\cite{Barndorff:exponential_families}.
  For self-containedness we provide a proof in our notation in Appendix~\ref{S-sec:proof2}.
\end{proof}
  The maximum $p^{*}$ in Theorem~\ref{thm:gmle} is called the \emph{extended} maximum likelihood estimate (EMLE).
By Theorem~\ref{thm:gmle}, when $F_{t}$ is known, the EMLE $p^{*}$ can be computed by computing the MLE
on~$\Ecal_{F_{t},A}$.  If the MLE $\theta^{*}$ exists, then $p^{*}=p_{\theta^{*}}$.

\subsection{Decomposable models}
\label{sec:reducible-models}

Computing $\Fbf_{t}$ or finding an approximation is easier when the simplicial complex $\Delta$ of the model is
decomposable.  We need the following definitions.

Let $V'\subset V$.  The \emph{restriction} or \emph{induced subcomplex} to $V'$ is
$\Delta|_{V'}=\{ S\in\Delta\mid S\subseteq V'\}$.  The subcomplex $\Delta|_{V'}$ is \emph{complete}, if $\Delta|_{V'}$
contains~$V'$ (and thus all subsets of~$V'$).  In this case, we also say that $V'$ is \emph{complete} in~$\Delta$.

A subset $S\subset V$ is a \emph{separator} of $\Delta$ if there exist $V_{1},V_{2}\subset V$ with $V_{1}\cap V_{2}=S$,
$\Delta=\Delta|_{V_{1}}\cup\Delta|_{V_{2}}$ and $V_{1}\neq S\neq V_{2}$.
A simplicial complex that has a complete separator is called \emph{reducible}. By extension, we also call the
hierarchical model reducible.

A hierarchical model is \emph{decomposable} if its generating set is a union
$\Delta=\Delta_{1}\cup\Delta_{2}\cup\dots\Delta_{r}$ of induced sub-complexes $\Delta_{i}=\Delta|_{V_{i}}$ in such a way
that
\begin{enumerate}
\item each $\Delta_{i}$ is a complete simplex: $\Delta_i = \{ S \subseteq V_{i}\}$; and
\item $(\Delta_{1}\cup\dots\cup\Delta_{i})\cap\Delta_{i+1}$ is a complete simplex.
\end{enumerate}
In other words, $\Delta$ arises by iteratively gluing simplices along complete sub-simplices.

Lemma~\ref{lem:reducible-facial} below states that, if $\Delta$ is reducible, then any facial set for $\Delta$ is the
intersection of the preimage of facial sets for its components.  It is a simple reformulation of Lemma~8 in
\citep{EFRS06:Polyhedral_conditions_MLE}.

\begin{lemma}
  \label{lem:reducible-facial}
  Let $\Delta$ be reducible into two components  $\Delta|_{V_{1}}$ and~$\Delta|_{V_{2}}$.
  \begin{enumerate}
  \item If $F\subseteq I$ is facial with respect to~$\Delta$, then $\pi_{V_{1}}(F)$ and $\pi_{V_{2}}(F)$ are facial with
    respect to~$\Delta|_{V_{1}}$ and~$\Delta|_{V_{2}}$.
  \item Conversely, if $F_{1}\subseteq I_{V_{1}}$ and $F_{2}\subseteq I_{V_{2}}$ are facial with respect
    to~$\Delta|_{V_{1}}$ and~$\Delta|_{V_{2}}$, then $\pi_{V_{1}}^{-1}(F_{1})\cap\pi_{V_{2}}^{-1}(F_{2})$ is facial with
    respect to~$\Delta$.
  \end{enumerate}
  Thus, for any $T\subseteq I$, let $T_{1}=\pi_{V_{1}}(T)$ and~$T_{2}=\pi_{V_{2}}(T)$.  Then
    \begin{equation*}
      \face_{\Delta}(T) = \pi_{V_{1}}^{-1}(\face_{\Delta|_{V_{1}}}(T_{1})) \cap \pi_{V_{2}}^{-1}(\face_{\Delta|_{V_{2}}}(T_{2})).
    \end{equation*}
\end{lemma}

Lemma~\ref{lem:reducible-facial} generalizes to more than one separator and thus to more than two components.  It
becomes particularly simple when these components are complete: in that case, $\face_{\Delta|_{V_{1}}}(T_{1})=T_1$.
Taking the preimage we obtain
\begin{equation*}
\pi_{V_1}^{-1}(\pi_{V_1}(T))=\big\{i\in I:\;\exists i'\in T\text{ such that }\pi_{V_1}(i)=\pi_{V_1}(i')\big\} \supseteq T.
\end{equation*}
Thus, for a decomposable complex $\Delta=\Delta_{1}\cup\Delta_{2}\cup\dots\cup\Delta_{r}$,
we have
\begin{equation}
  \label{eq:decomposable}
  \face_{\Delta}(T) = \pi_{1}^{-1}(\pi_{1}(T)) \cap \pi_{2}^{-1}(\pi_{2}(T)) \cap\dots\cap \pi_{r}^{-1}(\pi_{r}(T))
\end{equation}
for any $T\subseteq I$, where $\pi_{i}=\pi_{V(\Delta_{i})}$.

\section{Approximations of facial sets}
\label{sec:facial-sets}

We consider a hierarchical model with simplicial complex $\Delta$ and marginal polytope~$\Pbf_{\Delta}$.  In this
section, we develop the details of our methodology to obtain inner and outer approximations to the facial set $F_t$ of
the data vector~$t$.

\subsection{Inner approximations}
\label{sec:inner-approximations}

To obtain an inner approximation, our strategy is to find a separator~$S$ of $\Delta$ and to complete it.  To be
precise, we augment~$\Delta$ by adding all subsets of~$S$.  Thus, we obtain a simplicial complex
$\Delta_{S}=\Delta\cup\{M : M\subseteq S\}$ in which $S$ is a complete separator.  We can apply
Lemma~\ref{lem:reducible-facial} to find the facial set $\face_{\Delta_{S}}(I_{+})$, and this will be our inner
approximation of~$\face_{t}$, because $\face_{\Delta_{S}}(I_{+})\subseteq\face_{\Delta}(I_{+})  =\face_{t}$ according to
  Lemma~\ref{lem:sub-complex}.

An even simpler approximation is obtained by not only completing the separator itself, but also the two
parts~$V_{1},V_{2}$ separated by~$S$: the simplicial complex
$\Delta_{V_{1},V_{2}} := \{M : M\subseteq V_{1}\}\cup\{M:M\subseteq V_{2}\}$ is decomposable and contains~$\Delta$.  Its
facial sets can be computed from~\eqref{eq:decomposable}.

In general, the approximation obtained from a single separator (or, in general, a single super-complex) is not good;
that is, $F_{t}=\face_{\Delta}(I_{+})$ tends to be much larger than $\face_{\Delta_{S}}(I_{+})$ or
$\face_{\Delta_{V_{1},V_{2}}}(I_{+})$.  Thus we need to combine information from several separators.  For example, given
two separators~$S,S'\subseteq V$, we find a chain of approximations
\begin{align*}
  G_{0}' &:= I_{+}, \; \\
  G_{1} := \face_{\Delta_{S}}(G_{0}')&, \;
  G_{1}' := \face_{\Delta_{S'}}(G_{1}), \;\\
  G_{2} := \face_{\Delta_{S}}(G_{1}')&, \;
  G_{2}' := \face_{\Delta_{S'}}(G_{2}), \;\\
  &\;\vdots
\end{align*}
that satisfy
\begin{equation*}
  I_{+}\subseteq G_{1} \subseteq G_{1}' \subseteq G_{2} \subseteq \dots \subseteq F_t,
\end{equation*}
where all inclusions except the last one are due to the definition of $\face_{\Delta_{S}}(T)$ or $\face_{\Delta_{S'}}(T)$  as the smallest facial sets containing $T$ in $\Delta_S$ or $\Delta_{S'}$. The last inclusion is a consequence of Lemma~\ref{lem:sub-complex} since both $\Delta_S$ and $\Delta_{S'}$ contain~$\Delta$.
This chain of approximations has to stabilize; that is, after a certain number of iterations, the
approximations will not improve any more.  The limit $\face_{S,S'}(I^{+}) := \bigcup_{i}G_{i} = \bigcup_{i}G_{i}'$ can
be characterized as the smallest subset of~$I$ that contains~$I^{+}$ and is facial both with respect to~$\Delta_{S}$
and~$\Delta_{S'}$.  The same iteration can be done replacing $\Delta_{S}$ and $\Delta_{S'}$ by $\Delta_{V_{1},V_{2}}$
and $\Delta_{V'_{1},V'_{2}}$.  Applying in turn $\face_{\Delta_{V_{1},V_{2}}}$ and $\face_{\Delta_{V'_{1},V'_{2}}}$
gives another approximation~$\tilde\face_{S,S'}(I^{+})$, namely the smallest subset of~$I$ that contains~$I^{+}$ and is
facial both with respect to~$\Delta_{V_{1},V_{2}}$ and~$\Delta_{V_{1}',V_{2}'}$. This latter approximation will be used
in Section~\ref{sec:4x4}.  Clearly, $\tilde\face_{S,S'}(I^{+})$ is a worse approximation than~$\face_{S,S'}(I^{+})$,
since $\tilde\face_{S,S'}(I^{+})\subseteq\face_{S,S'}(I^{+})\subseteq F_{t}$, but it is easier to compute.

We use the following strategies:
\begin{enumerate}
\item if possible, use all separators of a graph.
\label{strat:all-separators}
\end{enumerate}
We illustrate this strategy in Section~\ref{sec:NLTCS-simulation}, using a graphical model associated with the
NLTCS data set.

There are two problems with this strategy: First, if $S$ is such that either $V_{1}$ or $V_{2}$ is large, then it
becomes difficult to compute $\face_{\Delta|_{V_{1}}}$ and $\face_{\Delta|_{V_{2}}}$.  Such
``bad'' separators always exist: namely, each node $i\in V$ is separated by its neighbours from all other nodes.  In
this case, $V_{1}$ consists of $i$ and its neighbours, and $V_{2}$ consists of $V\setminus\{i\}$.  For such a ``bad''
separator we can only compute $\face_{\Delta_{V_{1},V_{2}}}$, but not~$\face_{\Delta_{S}}$.  Second, the number of
separators may be large.
Thus, when computing the inner approximation, it may take a long time until the iteration over all separators converges.
A faster alternative strategy is the following:
\begin{enumerate}[resume]
\item use all separators such that both $V_{1}\setminus S$ and $V_{2}\setminus S$ are not too small (for example,
  $\min\{|V_{1}\setminus S|,|V_{2}\setminus S|\}\ge 3$).
  \label{strat:good-separators}
\end{enumerate}

In the case of the grids studied in Sections~\ref{sec:4x4} and~\ref{sec:5x10}, which have a lot of regularity, we use an
adapted strategy:
\begin{enumerate}[resume]
\item in a grid, use the horizontal, vertical and diagonal separators.
\label{strat:regular-separators}
\end{enumerate}
In the case of grids, the vertical separators form a family of pairwise disjoint separators.  In
Section~\ref{sec:large-graphs} we show how to make use of such a family to study faces of hierarchical models, even
if the facial sets are so large that they become computationally intractable.

\subsection{Outer approximations}
\label{sec:outer-approximations}

By Lemma~\ref{lem:sub-complex}, the facial set $\face_{\Delta'}(S)$ for a simplicial sub-complex
$\Delta'\subseteq\Delta$ provides an outer approximation of~$\face_\Delta(S)$.  Removing sets from~$\Delta$ decreases
the dimension of the marginal polytope, so it is often easier to compute $\face_{\Delta'}(S)$ than to
compute~$\face_\Delta(S)$.  Our main strategy is to look at induced sub-complexes.

When comparing $\Delta$ with an induced sub-complex~$\Delta|_{V'}$ for some $V'\subset V$, we have to be precise about whether we consider $\Delta|_{V'}$ as a complex
on $V$ or on $V'$.  When we consider it on $V$, then its design matrix ${A}$ has columns $f_{i}$ indexed by~$i\in I$.
When we consider it on $V'$, its design matrix $A'$ has columns $f'_{i}$ indexed by~$I_{V'}$.
Because we have the same set of interactions whether we are on $V$ or $V'$, we have for $i\in I$ and $i'\in
I_{V'}$,
\begin{equation}
\label{a=b}
f_i=f'_{i'} \Leftrightarrow  i\in \pi_{V'}^{-1}(i').
\end{equation}
Therefore the marginal polytopes of the two models are the same since they are the convex hull of the same set of vectors $\{f_i, i\in I\}=\{f'_{i'}, i'\in I_{V'}\}$. The relationship between the facial sets on $V$ and $V'$ is as follows:
\begin{lemma}
  \label{lem:facial-subcomplex}
  Let $V'\subseteq V$.  For $K\subset I$, we have 
  $$\face_{\Delta|_{V'}}(K) = \pi_{V'}^{-1}(\face'_{\Delta|_{V'}}(\pi_{V'}(K))).$$
  Here, $\face'_{\Delta|_{V'}}$ denotes the facial set when $\Delta|_{V'}$ is considered as a simplicial complex on~$V'$,
  and $\face_{\Delta|_{V'}}$ denotes the facial set when $\Delta|_{V'}$ is considered as a simplicial complex on~$V$.
\end{lemma}
\begin{proof}
For $K\subset I$, the two sets ${\cal A}=\{a_i, i\in K\}$ and ${\cal B}=\{b_{i'}, i'\in \pi_{V'}(K)\}$ are identical and therefore the smallest faces of the marginal polytopes for $\Delta_{V'}$  on $V$ or $V'$ containing ${\cal A}$ and ${\cal B}$ respectively are the same. 

By definition of $F'_{\Delta_{V'}}(\pi_{V'}(K))$, the smallest face containing $\cal{B}$ is defined by $\{b_{i'}, i'\in F'_{\Delta_{V'}}(\pi_{V'}(K))\}$. By definition of $F_{\Delta_{V'}}(K)$, the smallest face containing $\cal{A}$ is $\{a_i, i\in F_{\Delta_{V'}}(K)\}$. 
Also by \eqref{a=b}, we have that 
$\{a_i, \;i\in \pi_{V'}^{-1}(F'_{\Delta_{V'}}(\pi_{V'}(K)))\}=\{b_{i'},\;{i'}\in F'_{\Delta_{V'}}(\pi_{V'}(K))\}.$
Therefore
$F_{\Delta_{V'}}(K)=\pi_{V'}^{-1}(F'_{\Delta_{V'}}(\pi_{V'}(K))).$
\end{proof}


\medskip%
In general, $\face_{\Delta|_{V'}}(I_{+})$ is not a good approximation of~$\face_{\Delta}(I_{+})$.  This
approximation can be improved by considering several subsets of~$V$.  To be precise, if $V_{1},\dots,V_{r}\subseteq V$, then
$\face_{\Delta}(I_{+})\subseteq\face_{\Delta|_{V_{i}}}(I_{+})$ for $i=1,\dots,r$, and thus
$\face_{\Delta}(I_{+})\subseteq\bigcap_{i=1}^{r}\face_{\Delta|_{V_{i}}}(I_{+})
=:\face_{V_{1},\dots,V_{r};\Delta}(I_{+})$.
In contrast to the case of the inner approximation, no repeated iteration is needed.  Thus, the outer
approximation is faster to compute.

The question is now how to choose the subsets~$V_{i}$.  Clearly, the subsets $V_{i}$ should cover~$V$, and, more
precisely, they should cover~$\Delta$, in the sense that for any~$D\in\Delta$ there should be one $V_{i}$
with~$D\in \Delta|_{V_{i}}$.  The larger the sets~$V_{i}$, the better the approximation becomes, but the more difficult it
is to compute $\face_{V_{1},\dots,V_{r};\Delta}(I_{+})$.  One generic strategy is the following:
\begin{enumerate}
\item use all subsets of $V$ of fixed cardinality~$k$ plus all facets $D\in\Delta$ with~$|D|\ge k$.
\label{strat:fixed-cardinality}
\end{enumerate}
This choice of subsets indeed covers~$\Delta$.  The parameter $k$ should be chosen as large as possible such that
computing $\face_{V_{1},\dots,V_{r};\Delta}(I_{+})$ is still feasible.  Note that computing $\face_{\Delta|_{D}}(I_{+})$
for $D\in\Delta$ is trivial, since $\Pbf_{\Delta|_{D}}$ is a simplex.  Another natural strategy, due
to~\cite{MassamWang15:local_approach}, is the following:
\begin{enumerate}[resume]
\item for fixed $k$, use balls $B_{k}(v)=\{w: d(v,w)\le k\}$ around the nodes $v\in V$, where $d(\cdot,\cdot)$ denotes
  the edge distance in the graph.
\label{strat:neighbourhoods}
\end{enumerate}

Our general philosophy is that the subsets $V_{i}$ should be large enough to preserve some of the structure
of~$\Delta$.  For example, for the grid graphs, we suggest to use $3\times3$ sub-grids.  These graphs have two nice
properties: First, they already have the appearance of a small grid.  Second, for any vertex $v\in V$, there is a
$3\times3$ sub-grid that contains $v$ and all neighbours of~$v$.  We will compare two different strategies:
\begin{enumerate}[resume]
\item for a grid, use all $3\times3$ sub-grids;
\label{strat:all-subgrids}
\item cover a grid by $3\times3$ sub-grids.
\label{strat:subgrid-cover}
\end{enumerate}
In Section~\ref{sec:5x10} we compare these two methods, and we observe that, in the example of the $5\times10$ grid, it
suffices to only look at a covering.

In general, it is not enough to look at induced sub-complexes, unless $\Delta$ has a complete separator (see
Section~\ref{sec:reducible-models}).  However, the approximation tends to be good and gives the correct facial set in
many cases.

\subsection{Comparing the two approximations}
\label{sec:comparing-approximations}

Suppose that we have computed two approximations $F_{1},F_{2}$ of $F_{t}$ such that $F_{1}\subseteq F_{t}\subseteq
F_{2}$.  If we are in the lucky case that $F_{1}=F_{2}$, then we know that $F_{t}=F_{1}=F_{2}$.  In general, the
cardinality of $F_{2}\setminus F_{1}$ indicates the quality of our approximations.

The sets $F_{1}$, $F_{2}$ and $F_{t}$ can also be compared by the ranks of the matrices $\tilde A_{F_1}$, $\tilde A_{F_2}$ and
$\tilde A_{F_{t}}$ obtained from $ \tilde{A}$ by keeping only the columns indexed by $F_1$, $F_2$ and $F_{t}$,
respectively.  Clearly, $\rank\tilde A_{F_{1}}\le\rank\tilde A_{F_{t}}\le\rank\tilde A_{F_{2}}$.  Note that
$\rank\tilde A_{F_{2}}-1$ equals the dimension of the corresponding face~$\Fbf_{2}$ of~$\Pbf$, and $\rank \tilde{A}_{F_{t}}-1$
equals the dimension of~$\Fbf_{t}$.
Although $F_{1}$ does not necessarily correspond to a face of~$\Pbf$, we
can bound the codimension of~$\Fbf_{t}$ in~$\Fbf_{2}$ by
\begin{equation*}
  \dim\Fbf_{2}-\dim\Fbf_{t} \le \rank\tilde A_{F_{2}}-\rank\tilde A_{F_{1}}.
\end{equation*}
In particular, if $\rank\tilde A_{F_{2}}=\rank\tilde A_{F_{1}}$, then we know that $F_{t}=F_{2}$.  In this case, our approximations
give us a precise answer, even if $F_{1}\neq F_{2}$ and the lower approximation $F_{1}$ is not tight.

\section{Parameter Estimation when the MLE does not exist}
\label{sec:parameter-estimation}

\subsection{Computing the extended MLE}
\label{sec:computingemle}
If the MLE $\theta^{*}$ exists, then it can be computed by finding the unique maximum of the log-likelihood
function~$l(\theta)$ given in~\eqref{lik-theta}.  As mentioned before, $l(\theta)$ is concave (or even strictly concave,
if the parameters~$\theta$ are identifiable), and thus the maximum is, at least in principle, easy to find (in practice,
for larger models, it may be difficult to evaluate the function~$k(\theta)$, which involves a sum over~$I$; but we will
not discuss this problem here).  In general, the maximum cannot be found symbolically, but there are efficient numerical
algorithms to maximize concave functions.  Any reasonable hill-climbing algorithm should be capable of finding the MLE.
An example of an algorithm commonly used is \emph{iterative proportional fitting} (IPF), which can be thought of as an
algorithm of Gauss-Seidel type~\citep{CsiszarShields04:Information_Theory_and_Statistics}.

When the MLE does not exist but the facial set $F=F_{t}$ of the data is known, then it is straight forward to compute
the extended MLE~$p^{*}$.  In this case, we know that $p^{*}$ lies in~$\Ecal_{F,A}$.  To find $p^{*}$, we need to
optimize the log-likelihood $\tilde l$ over~$\Ecal_{F,A}=\{p_{F,\theta}:\theta\in\R^{h}\}$.  Plugging the
parametrization $p_{F,\theta}$ (see Theorem~\ref{thm:closure-of-expfam}) into $\tilde l$ tells us that we need to
optimize the restricted log-likelihood function
\begin{equation}
  \label{eq:likelihood-on-face}
  l_{F}(\theta) = \log(\prod_{i\in I_{+}}p_{F,\theta}(i)^{n(i)}) = \sum_{j\in J}\theta_{j}t_{j} - N k_{F}(\theta).
\end{equation}
This problem is of a similar type as the problem to maximize $l$ in the case that the MLE exists, and the same
algorithms as discussed above can be used.  The problem here is slightly easier, since $F$ is smaller than~$I$.
However, 
as stated above, the parametrization $\theta\mapsto p_{F,\theta}$ is never identifiable.  Of course, this problem is
easy to solve by selecting a set of independent parameters among the~$\theta_{j}$.
However, depending on the choice of the independent subset, the values of the
parameters change, and in particular, it is meaningless to compare the values of the parameters $\theta_{j}$ with
parameter values of any other distribution in $\Ecal_{A}$ or in the closure~$\ol{\Ecal_{A}}$.

Before explaining how to find better parameters on~$\Ecal_{F,A}$, let us discuss what happens if the facial set $F_{t}$
of the data is not known.  As mentioned before, whether or not the MLE exists, the log-likelihood function
$l(\theta)$ 
is always strictly concave (assuming that the parametrization is identifiable).  When the MLE does not exist, then the
maximum is not at a finite value~$\theta^{*}$, but lies ``at infinity.''  Still, as observed
by~\citet[Section~3.15]{Geyer09:Likelihood_inference_in_exponential_families}, any reasonable numerical
``hill-climbing'' algorithm that tries to maximize the likelihood will tend towards the right direction.  Such a numeric
algorithms generates a sequence of parameter values $\theta^{(1)},\theta^{(2)},\theta^{(3)},\dots$ with increasing
log-likelihood values~$l(\theta^{(1)})\le l(\theta^{(2)})\le\dots$.  Since $l(\theta)$ is concave, our optimization
problem is numerically easy (at least in theory), and for any reasonable such algorithm, the limit
$\lim_{s\to\infty}l(\theta^{(s)})$ will equal $\sup_{\theta}l(\theta) = \max_{p\in\ol{\Ecal_{A}}}\tilde l(p)$.
The algorithm will stop when the difference $l(\theta^{(s+1)})-l(\theta^{(s)})$ becomes negligeably small.  The output,
$\theta^{(s)}$, then gives a good approximation of the EMLE, in the sense that $p^{*}$ and $p_{\theta^{(s)}}$ are close
to each other.
For many applications, such as in machine learning, where it is more important to have good values of the parameters
instead of trying to model the ``true underlying distribution,'' or when doing a likelihood test, where the value of the
likelihood is more important than the parameter values, this may be good enough.

However, in this numerical optimization, some of the parameters $\theta_{j}$ will tend to $\pm\infty$, which may lead to
numerical problems.  For example, it may happen that one parameter goes to~$+\infty$ and a second parameter to~$-\infty$
in such a way that their sum remains finite (see Appendix~\ref{S-sec:two-binaries} for a simple such example with two variables).  This implies that a difference between two large numbers has to computed,
which is numerically unstable.  Also, it is not clear, which parameters tend to infinity numerically.  In fact, this may
depend on the chosen algorithm; i.e. different algorithms may yield approximations of the EMLE that are qualitatively
different in the sense that different parameters diverge. 

To avoid such problems, we propose a change of coordinates that allows us to control which parameters diverge, at least
in the case where we know the facial set~$F_{t}$.  If $F_{t}$ is unknown, but if we know approximations
$F_{1}\subseteq F_{t}\subseteq F_{2}$, we can use this knowledge to identify some parameters that definitely remain
finite, while some parameters definitely diverge.  We cannot control the behaviour of the remaining parameters, but, as will be illustrated in Section~\ref{sec:NLTCS-simulation}, the MLE obtained with the model on $\Fbf_2$ lies closer to the EMLE than the MLE on the original model.
The more information we have about the facial set~$F_{t}$, the better we can control the above
mentioned pathologies.

\subsection{An identifiable parametrization}
\label{sec:param-trafos}

We have seen that when we use the parametrization $\theta\mapsto p_{F_{t},\theta}$ of $\Ecal_{A,F_{t}}$ in the case
where $F_{t}\neq I$, we have to expect the following (interrelated) issues:
\begin{enumerate}
\item The parametrization is not identifiable, i.e.~there are parameters~$\theta,\theta'$ with $p_{F_{t},\theta}=p_{F_{t},\theta'}$.
\item While the parametrization $\theta\mapsto p_{F_{t},\theta}$ of $\Ecal_{F_{t},A}$ looks similar to the parametrization $\theta\mapsto
  p_{\theta}$ of~$\Ecal_{A}$, the values of the parameters in both parametrizations are not related to each other.
\item When $p_{\theta^{(s)}}\to p_{F_{t},\theta}$ as $s\to\infty$ for some parameter values $\theta^{(s)},\theta$, then
  some of the parameter values $\theta^{(s)}$ diverge to~$\pm\infty$.  When computing probabilities, there may be
  linear combinations of these diverging parameters that remain finite.
\end{enumerate}
Next we show that if $F_{t}$ is known, then, with a convenient choice of~$L$, the parameters $\mu_{L}$ (introduced in
Section~\ref{sec:discr-expon-famil}) solve 1 and~2 and improve~3.  Afterwards, we discuss what can be done if $F_{t}$
is not known.  We briefly discuss the general solution towards~3 in Appendix \ref{S-sec:best-parameters}.
In any case, the choice of the parameters will depend on the facial set~$F_{t}$:  it is not possible to define a
single parametrization that works for all facial sets simultaneously.

Suppose that $F_{t}$ is known.  We choose a zero element in~$I_{+}$ and consider the parameters $\mu_{i}$ as in
Section~\ref{sec:mus-0}.  Recall that
\begin{equation*}
  \mu_{i}(\theta) = \<\t, f_{i}-f_{0}\> = \log p(i)/p(0), i\in I.
\end{equation*}
As mentioned in Section~\ref{sec:mus-0}, the parameters $\mu_{i}$ are not independent, and we need to choose an independent subset~$L$.  We will do this in two steps.
\begin{enumerate}
\item Choose a maximal subset~$L_{t}$ of $F_{t}$ such that the parameters~$\mu_{i}$, $i\in L_{t}$ are independent.
\item Then extend $L_{t}$ to a maximal subset $L\subseteq I$ such that the parameters~$\mu_{i}$, $i\in L$ are
  independent by adding elements $i\in I\setminus F_{t}$.
\end{enumerate}
It follows from Theorem~\ref{thm:gmle} that the following holds.
\begin{enumerate}
\item The subset $\mu_{i}$, $i\in L_{t}$, of the parameters $\mu_{L}$ gives an identifiable parametrization
  of~$\Ecal_{F_{t},A}$.
\item Let $\mu^{*}_{i}$, $i\in L_{t}$, be the parameter values that maximize~$l_{F_{t}}$ (and thus give the EMLE).  When
  the likelihood $l(\mu)$ in~\eqref{lik-mu} is maximized numerically on~$I$, then in successive iterations of the
  maximization, the estimates $\mu^{(s)}_i$ are such that
  \begin{equation*}
    \mu^{(s)}_{i} \to
    \begin{cases}
      \mu^{*}_{i}, & i\in L_{t}, \\
      -\infty, & \text{ otherwise.}
    \end{cases}
  \end{equation*}
  In particular, no parameter tends to~$+\infty$.
\end{enumerate}
The last property ensures a consistency of the parameters $\mu_{i}$ on $\Ecal_{A}$ and on~$\Ecal_{F_{t},A}$.  This is
important in those cases where the parameters have an interpretation and where it is of interest to know the value of
some parameters, if it is well-defined.  For example, in hierarchical models, the parameters correspond to
``interactions'' of the random variables, and it may be of interest to know, which of these interactions are important.
Thus, it is of interest to know the size of the corresponding parameter.  Usually, it is not the parameter~$\mu_{i}$,
but the original parameters~$\theta_{i}$ that have an interpretation.  But when we understand the parameters~$\mu_{i}$,
we can also tell which of the paramters~$\theta_{i}$ or which combinations of the parameters $\theta_{i}$ have finite
well-defined values and can be computed, and which parameters diverge:

\begin{lemma}
  \label{lem:mu-properties}
  Suppose that $\theta^{(s)}$, $s\in\N$, are parameter values such that $p_{\theta^{(s)}}\to p^{*}$ as $s\to\infty$.
  For any $i\in L_{t}$, the linear combination
  \begin{equation*}
    \mu_{i}^{(s)} = \<\theta^{(s)},f_{i}\> 
  \end{equation*}
  has a well-defined finite limit as $s\to\infty$.  Any linear combination of the $\theta_{i}^{(s)}$ that has a
  well-defined finite limit (that is, a limit that is independent of the choice of the sequence~$\theta^{(s)}$) is
  a linear-combination of the~$\mu_{i}^{(s)}$ with $i\in L_{t}$.
\end{lemma}
\begin{proof}
  The first statement follows from
  \begin{equation*}
    \mu_{i}^{(s)}=\log p_{\theta^{(s)}}(i)/p_{\theta^{(s)}}(0)\to\log
    p^{*}(i)/p^{*}(0).
\end{equation*}
  For the second statement, note that any linear combination of the $\theta$ is also a linear
  combination of the~$\mu$, since the linear map $\theta\mapsto\mu(\theta)$ is invertible.  We now show that if a linear
  combination $\sum_{i}a_{i}\mu_{i}$ involves some $\mu_{j}$ with $j\notin L_{t}$, then there exist sequences
  $\mu^{(s)}$, $\mu^{\prime(s)}$ of parameters with
  \begin{equation*}
    \lim_{s\to\infty}p_{\mu^{(s)}}=\lim_{s\to\infty}p_{\mu^{\prime(s)}}
    \quad\text{and}\quad
    \lim_{s\to\infty}\sum_{i}a_{i}\mu_{i}^{(s)}\neq\lim_{s\to\infty}\sum_{i}a_{i}\mu_{i}^{\prime(s)}.
  \end{equation*}
  So suppose that $\mu^{(s)}$ is a sequence of parameters such that $\lim_{s\to\infty}p_{\mu^{(s)}}$ exists and such
  that $\lim_{s\to\infty}\sum_{i}a_{i}\mu_{i}^{(s)}$ is finite.  Define
  \begin{equation*}
    \mu_{i}^{\prime(s)} =
    \begin{cases}
      \mu_{j}^{(s)} + 1, & \text{ if i=j}, \\
      \mu_{i}^{(s)}, & \text{ otherwise.}
    \end{cases}
  \end{equation*}
  An easy computation shows that
  \begin{equation*}
    \lim_{s\to\infty}p_{\mu^{\prime(s)}}=\lim_{s\to\infty}p_{\mu^{(s)}}
    \quad\text{and}\quad
    \lim_{s\to\infty}\sum_{i}a_{i}\mu_{i}^{\prime(s)}=\lim_{s\to\infty}\sum_{i}a_{i}\mu_{i}^{(s)} + a_{j}.
    \qedhere
  \end{equation*}
\end{proof}

Suppose now that we do not know~$F_{t}$, but that instead we have  approximations $F_{1}$, $F_{2}$ that satisfy
\begin{equation*}
  I_{+}\subseteq F_{1} \subseteq F_{t} \subseteq F_{2} \subseteq I.
\end{equation*}
In this case, we proceed as follows to obtain an independent subset~$L$ among the parameters~$\mu_{i}$:
\begin{enumerate}
\item Choose a maximal subset~$L_{1}$ of $F_{1}$ such that the parameters~$\mu_{i}$, $i\in L_{1}$ are independent.
\item Then extend $L_{1}$ to a maximal subset $L_{2}\subseteq F_{2}$ such that the parameters~$\mu_{i}$, $i\in L_{2}$
  are independent by adding elements $i\in F_{2}\setminus F_{1}$.
\item Finally, extend $L_{2}$ to a maximal subset $L\subseteq I$ such that the parameters~$\mu_{i}$, $i\in L$ are
  independent by adding elements $i\in I\setminus F_{2}$.
\end{enumerate}
These parameters have the following properties that follow directly from Lemma~\ref{lem:mu-properties}:
\begin{lemma}
  Suppose that $\theta^{(s)}$, $s\in\N$, are parameter values such that $p_{\theta^{(s)}}\to p^{*}$ as $s\to\infty$, and
  let $\mu_{i}^{(s)} = \<\theta^{(s)},f_{i}\>$.
  \begin{enumerate}
  \item For any $i\in L_{1}$, the linear combination
    \begin{equation*}
      \mu_{i}^{(s)} = \<\theta,f_{i}\>
    \end{equation*}
    has a well-defined finite limit as $s\to\infty$.  Thus, any linear combination of the $\mu_{i}^{(s)}$ with $i\in
    L_{1}$ has a well-defined limit as $s\to\infty$.
  \item Any linear combination $\sum_{i}a_{i}\mu_{i}^{(s)}$ that has a well-defined limit as $s\to\infty$ is in fact a
    linear combination of the $\mu_{i}^{(s)}$ with~$i\in L_{2}$.  Thus, a linear combination that involves at
    least one $\mu_{j}^{(s)}$ with $j\in L\setminus L_{2}$ does not have a well-defined limit.
  \end{enumerate}
\end{lemma}

\section{Simulation study and applications to real data}
\label{sec:experiments}

In this section, we illustrate our methodology.  In~\ref{sec:4x4}, we simulate data for the graphical model of the
$4\times4$ grid and show how to exploit the various types of separators in order to obtain good inner and outer
approximations.  We find that our method gives very accurate result in this model of modest size.
In~\ref{sec:NLTCS-simulation}, we work with the NLTCS data set, a real-world data set.
We compare different inner approximations~$F_{1}$ and find that most of the time, $F_{1}$ and $F_{2}$ are equal, and
thus they are both equal to~$F_{t}$.  We also compute the EMLE and compare these exact estimates to those obtained  when maximizing
the likelihood functions~$l$ and~$l_{F_{2}}$. We find the results given by $l_{F_{2}}$ better than those given by $l$, and extremely close to the finite components of the EMLE.

\subsection{\texorpdfstring{$4 \times 4$}{4x4} grid graph}
\label{sec:4x4}

We generated random samples of varying sizes for the graphical model of the $4\times4$ grid graph with binary variables (Fig.~\ref{4by4}).
For each sample, we compute inner and outer approximations $F_{1}$ and~$F_{2}$, and we compare them to the true facial
set~$F_{t}$, which we can obtain using linear programming.
To obtain an inner approximation, we use two strategies. Either, we iterate over all possible separators, of which there
are 106 (Strategy~\eqref{strat:all-separators} in Section~\ref{sec:inner-approximations}), or we iterate over the 3
horizontal, 3 vertical and 8 diagonal separators only (Strategy~\eqref{strat:regular-separators} in
Section~\ref{sec:inner-approximations}).  We obtain the same result with either strategy.  Clearly,
Strategy~\eqref{strat:regular-separators} is much faster.  To compute the outer approximation, we cover the $4\times4$
grid by four $3\times3$ grids (Strategy~\eqref{strat:all-subgrids} in Section~\ref{sec:outer-approximations}).

We generate samples from the hierarchical model $P_{\theta}$, where the vector of parameters
$\theta$ is drawn from a multivariate standard normal distribution (for each sample, new parameters were drawn).  The
results are given in Table~\ref{tab:4by4_std}.  For each sample size, \num{1000} samples were obtained.   Observe that the squared length
of the parameter vector $\theta$ is $\chi^2$-distributed with 40 degrees of freedom (since the number of parameters
is~40).  Thus, the expected length of $\theta$ is~40, which is large enough to move the distribution $p_{\theta}$ close
to the boundary of the model.  Indeed, we observed that when the MLE does not exist, the length of the numerical
estimate of the MLE vector is of the order of magnitude of~40 (see also the next example in
Section~\ref{sec:NLTCS-simulation}).
In all samples that we generated, $F_{t}=F_{2}$, and $F_{1}=F_{2}$ in the vast majority of cases.  Thus, for
this graph of relatively modest size, our approximations are very good.
We present additional simulation results in Appendix~\ref{S-sec:uniform-4x4}.

\begin{table} 
  \centering
  \caption{Facial set approximation of $4 \times 4$ grid graph(hierarchical log-linear model with parameters from standard normal distribution)}
  \begin{tabular}{dddd}
    \toprule
    \multicolumn{1}{c}{sample size} & \multicolumn{1}{c}{MLE does not exist} & \multicolumn{1}{c}{$F_1=F_t$} & \multicolumn{1}{c}{$F_2=F_t$} \\
    \midrule
    10 & 100.0\% & 97.7\% & 100.0\%\\
     50 & 89.5\% & 100.0\% & 100.0\% \\
    100 & 71.0\% & 100.0\% & 100.0\% \\
    150 & 52.0\% & 100.0\% & 100.0\% \\
    \bottomrule
  \end{tabular}
  \label{tab:4by4_std}
\end{table}

\subsection{NLTCS data set}
\label{sec:NLTCS-simulation}

To illustrate how approximate knowledge of the facial set allows us to say which parameters can be estimated (as
explained in Section~\ref{sec:parameter-estimation}), we study the NLTCS data set, which consists of \num{21574}
observations on 16 binary variables, called ADL1, \dots, ADL6, IADL1, \dots, IADL10.  The reader is referred to
\citet{DobraLenkoski11:Copula_Gaussian_graphical_models} for a detailed description of the data set.
To associate a hierarchical model to this data, we rely on the results of
\citet{DobraLenkoski11:Copula_Gaussian_graphical_models} who use a Bayesian approach to estimate the posterior inclusion
probabilities of edges.  We construct a graph by saying that $(x,y)$ is an edge if and only if the posterior inclusion
probability of $(x,y)$ is at least~\num{0.40}: we obtain Figure~\ref{NLTCS}.  Then we take the corresponding clique complex of this graph so that
our hierarchical model is a graphical model.
There are 314 parameters in this model, including up to 6-way interactions.
In total, the graph has 40 separators.
\begin{figure}[t]
  \centering
  \subfloat[\label{4by4}]{
    \begin{minipage}[c]{0.35\linewidth}
    \begin{tikzpicture}[scale=1.3] \grid{4}{4}  \end{tikzpicture}      
    \end{minipage}
  }
  \hfill
  \subfloat[\label{NLTCS}]{
    \begin{minipage}[c]{0.56\linewidth}
    \includegraphics[width=\textwidth]{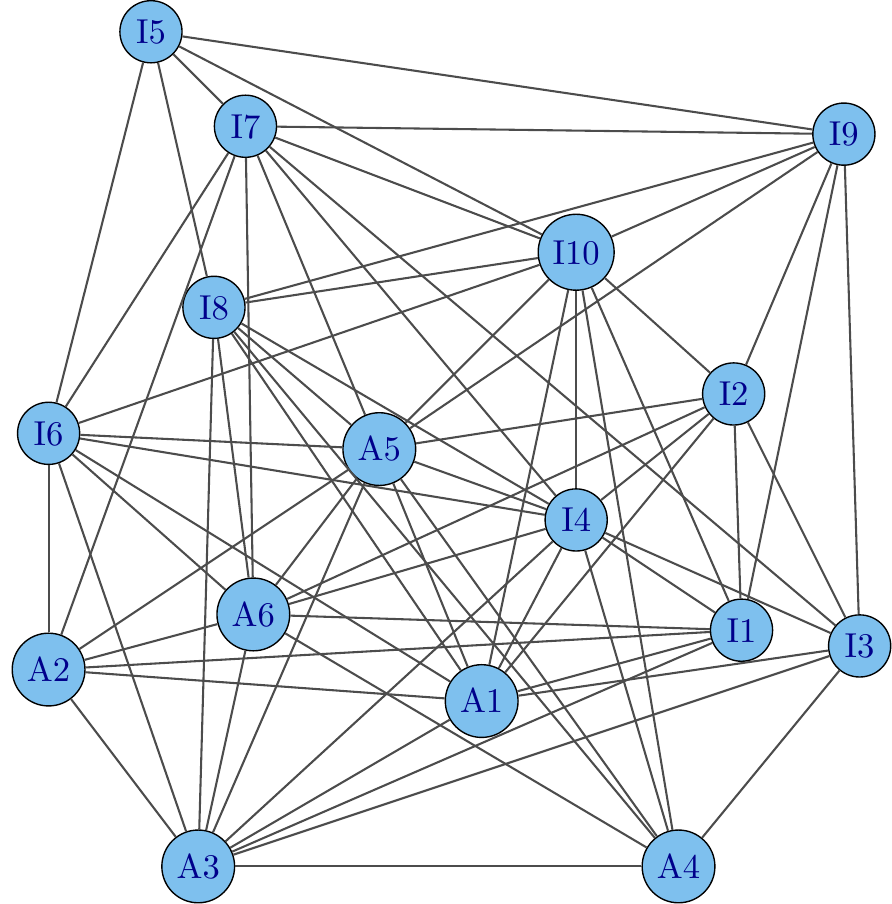}
  \end{minipage}
}
  \caption{\textup{(\protect\subref*{4by4})} $4 \times 4$ grid graph. \textup{(\protect\subref*{NLTCS})} Graphical model for NLTCS data set.  The label ``\textup{A}$n$'' abbreviates ADL$n$, ``\textup{I}$n$''
    abbreviates IADL$n$.}
\end{figure}

In order to compare the estimates obtained with or without worrying about the existence of the MLE and with or
without an approximation to~$F_t$, we maximize the loglikelihood given in terms of $\mu$, rather than $\theta$, as in
\eqref{lik-mu}. 
First we ignore the fact that the MLE might not exist and  numerically optimize the likelihood directly: we call this estimate $\hat{\mu}^{\MLE}$. Second, we find $F_t$ and compute the EMLE with
parameters denoted $\hat{\mu}^{\EMLE}$. Third, we obtain an inner and outer approximation to $F_t$ and consider the
resulting information on likelihood maximization. We call the resulting estimate~$\hmuFF$.
All estimates are computed using the Matlab function minFunc \citep{minFunc}.

To compute $\hat{\mu}^{\EMLE}$, we first compute the inner approximation $F_{1}$ that makes use of all the separators in
the graph (Strategy \ref{strat:all-separators} in Section~\ref{sec:inner-approximations}).  We also compute an outer
approximation $F_{2}$ from all $\binom{16}{5}=\num{4368}$ size five local models and the cliques of size six
(Strategy~\ref{strat:fixed-cardinality} in Section~\ref{sec:outer-approximations}).  We obtain $F_1=F_2$ and thus deduce
that $F_{t}=F_{1}=F_{2}$. We find $|F_t|=\num{49536}$, and so $|F_t^c|=2^{16}-\num{49536}=\num{16000}$. Therefore,
\num{16000} cell probabilities are zero in the EMLE, a precise estimate of those cells that we could not obtain from the MLE.  We obtain the EMLE by maximizing the loglikelihood function
$l_{F_{t}}$ as in~\eqref{eq:likelihood-on-face}.  Since $\rank(\tilde A_{F_t})=303$, the dimension of $\Fbf_t$ is 302, and
there are only 302 parameters in~$l_F$. This information is most important when testing the present model against another model ${\cal M}_2$ of smaller dimension. As pointed out by \cite{Geyer09:Likelihood_inference_in_exponential_families} and \cite{FienbergRinaldo12:MLE_in_loglinear_models}, the test statistic,  chi-square or  loglikelihood, has to be compared to the chi-square distribution with $302-d_2$ degrees of freedom, not $314-d_2$. Of course, for ${\cal M}_2$ also, $d_2$ is the dimension of the smallest face of the corresponding polytope containing the data.

To show how to use the inner and outer approximations when $F_t$ is not known, we construct coarser inner and outer
approximations to~$F_t$, respectively denoted $F_1'$ and~$F_2'$, and use them to compute another approximation
$\hat{\mu}^{F_2'\setminus F_1'}$ to the EMLE.  To compute $F'_1$, we just use 10 random separators.  We find $|F'_1|=\num{36954}$
and $\dim\Fbf'_{1}=\rank\tilde A_{F'_{1}}-1=300$.  To compute the outer approximation $F'_2$, we consider the $4368$ local size-five induced models and select among them the 1000 with the facial sets of smallest cardinality, which we glue together.
We find $|F'_2|=\num{50688}$ and $\dim\Fbf'_{2}=\rank\tilde A_{F'_{2}}-1=310$.  Thus, we know that at least $|I\setminus
F'_2|=2^{16}-\num{50688}=\num{14848}$ cell probabilities vanish in the EMLE. Since we pretend not to know $F_t$,
we replace $l_{F_t}$ by
\begin{equation}
  \label{eq:like4}
  l_{F'_2}(\mu)=\sum_{i \in I_{+}}\mu_i n(i) -N\sum_{i \in F'_2} \exp(\mu_i).
\end{equation}
We know that $\mu_i$ is estimable for $i\in F'_1$, that $\mu_i$ goes to negative infinity for $i\in F_2^{'c}$, and we cannot
say anything for $\mu_i$ with $i\in F'_2 \setminus F'_1$.

As explained in Section \ref{sec:param-trafos}, the components of $\mu$ are not functionally independent.  We choose
$L_{1}\subseteq F_{1}'$, $L_{2}\subseteq F_{2}'$ and $L\subseteq I$ as in Section~\ref{sec:param-trafos} (we note that
the zero cell belongs to~$I_{+}$).
Then any $\mu_{i}$, $i\in F'_{2}$, can be written as a linear combination of $\mu_{L_{2}}=(\mu_{i},i\in L_{2})$, and we
can write $\mu_i=\< b_i, \mu_{L_2}\>$ for an appropriate vector~$b_i$.  Thus, $l_{F'_{2}}(\mu)$ only depends on
$\mu_{L_{2}}=(\mu_{i},i\in L_{2})$, and~\eqref{eq:like4} can be rewritten as
\begin{equation}
  l_{F'_2}(\mu_{L_2})=\sum_{i \in I_{+}}\<b_i,\mu_{L_2}\> n(i) -N\sum_{i \in F'_2} \exp\<b_i,\mu_{L_2}\>.
  \label{eq:like5}
\end{equation}
 
Of course, the maximum of $l_{F_2'}$ does not exist but, as for the maximization of $l$, the computer still gives us a
numerical approximation, $\hat{\mu}_{L_2}$, and thus also a numerical estimate $\hat{\mu}_{i}=\langle
b_i,\hat{\mu}_{L_2}\rangle, i\in F_2'$.
In total, there are $|L_{2}|=\rank(\tilde A_{F'_2})-1=310$
independent parameters in the loglikelihood function~\eqref{eq:like5}.
Among them, there are $|L_{1}|=\rank(\tilde A_{F'_1})-1=300$ estimable parameters~$\mu_{i},i\in L_{1}$.  We cannot say anything
about the 10 parameters indexed by $L_2 \setminus L_1$.  If we know $F_t$, we can identify two more estimable
parameters. 

Table \ref{table: mle} gives the three estimates $\hat{\mu}_i^{\MLE},
\hat{\mu}_i^{\EMLE}$ and $\hmuFF_{i}$.  The naive estimator $\log \frac{n_i}{n_0}$ is also listed.
Estimates are listed for 19 arbitrarily chosen parameters among the 310 possible ones.
The first column of the table indicates whether the index $i$ belongs to $F_1'$, $F_{t}$ or~$F_{2}'$.
The second column lists the particular parameters considered.  By Theorem \ref{thm:gmle}, the only parameters $\mu_i$ with a finite estimate are those for $i\in F_t$.
This is illustrated in the  $\hat{\mu}^{\EMLE}$ column of Table \ref{table: mle}, with finite values for
$\hat{\mu}_i^{\EMLE}$, $i\in  F_t$ (green and pink rows), and infinite values for $\hat{\mu}_i^{\EMLE}$, $i\in I\setminus F_t$ (yellow and blue rows).  When we choose the coarser approximations, $F_1'$ and $F_2'$, to~$F_t$, we compute the estimate $\hmuFF$  using \eqref{eq:like5}. The components $\hmuFF_i$ indexed by $i\in F_t$ are excellent. They are finite and close to the corresponding components of~$\hat{\mu}^{\EMLE}$.
This can be seen by verifying numerically that the square length of  the projection on $F_t$ of the difference between $\hat{\mu}^{\MLE}$ and $\hat{\mu}^{\EMLE}$ is greater than that between
$\hmuFF$ and $\hat{\mu}^{\EMLE}$. Indeed, we have
\begin{equation*}
\lVert  \hmuFF_{F_t}-\hat{\mu}^{EMLE}_{F_t}\rVert^2 \approx 6.49 < 
\lVert \hat{\mu}^{MLE}_{F_t}-\hat{\mu}^{EMLE}_{F_t}\rVert^2 \approx 8.52\;. 
\end{equation*}
 The components  $\hmuFF_i$ indexed by $i\in F'_2\setminus F_t$   are finite while the corresponding components of $\hat{\mu}^{EMLE}$ are infinite but they are better than those of $\hat{\mu}^{MLE}$: numerically, we have
 \begin{equation*}
  \sum_{i\in\F'_{2}\setminus F_{t}}\left(\hmuFF_{i}\right)^{2} \approx \num{5184}
  > 
  \sum_{i\in\F'_{2}\setminus F_{t}}\left(\hat{\mu}^{\MLE}_{i}\right)^{2} \approx \num{4752}\;.
\end{equation*}
 The estimates $\hmuFF_i, i\in F'_2\setminus F_t$  are better than the corresponding $\hat{\mu}^{\MLE}_{i}$ since they are larger and thus ``closer to the truth''.
 For $i\in I\setminus F'_2$, corresponding to the blue rows of Table \ref{table: mle}, the components
 $\hmuFF_{i}$ are better than the $\mu^{\MLE}_{i}$ since, by construction, the $\hmuFF_{i}$ are infinite.

\begin{table}[t]
  \centering
  \caption{Parameter estimates from 3 methods compared with the relative frequency in the NLTCS data.  Here, each $i=(i_{1},\dots,i_{16})\in I=\{0,1\}^{16}$ is represented by the natural number~$\sum_{j=1}^{16}i_{j}2^{j-1}\in\{0,\dots,2^{16}-1\}$.}
  \begin{tabular}{cccccc} 
    \toprule 
    &           & naive estimate & \multicolumn{3}{c}{maximum likelihood estimates} \\ 
     \multicolumn{2}{r}{Parameter}
    &  $\log n_{i}/n_{0}$ &  $\hat{\mu}^{\MLE}_{i}$ & $\hat{\mu}^{\EMLE}_{i}$ & $\hmuFF_{i}$ \\
    \hline
    \rowcolor{green} $i\in F_{1}'$ &  $\mu_{512}$ 
    &  $-1.2472$  &  $-1.2482$ &  $-1.2482$  &  $-1.2482$ \\
    \rowcolor{green}               & $\mu_{65536}$  &  $-1.7644$  &  $-1.7976$ &  $-1.7975$  &  $-1.7975$ \\ 
    \rowcolor{green}               & $\mu_{16}$    &  $-2.3958$  &  $-2.3844$ &  $-2.3846$  &  $-2.3846$ \\ 
    \rowcolor{green}               & $\mu_{528}$   &  $-2.5429$  &  $-2.6504$ &  $-2.6504$  &  $-2.6504$ \\ 
    \rowcolor{green}               & $\mu_{2048}$  &  $-2.8813$  &  $-2.7246$ &  $-2.7243$  &  $-2.7243$ \\ 
    \hline
    \rowcolor{LRed} $i\in F_{t}\setminus F_{1}'$ &  $\mu_{32960}$ &  $-\infty$ & $-13.8205$ & $-13.8207$  & $-13.8205$  \\ 
    \rowcolor{LRed}                & $\mu_{34881}$ & $-\infty$ & $-14.3693$ & $-14.3693$ &  $-14.3692$  \\ 
    \hline
    \rowcolor{yellow} $i\in F_{2}'\setminus F_{t}$ & $\mu_{36864}$ & $-\infty$  & $-30.8729$  &    $-\infty$ & $-34.9805$ \\ 
    \rowcolor{yellow} & $\mu_{36880}$ & $-\infty$ & $-39.6536$ &   $-\infty$  & $-45.2229$ \\  
    \rowcolor{yellow} & $\mu_{388}$   & $-\infty$ & $-28.9090$ &    $-\infty$ &  $-29.4525$ \\ 
    \rowcolor{yellow} & $\mu_{32769}$ & $-\infty$ & $-32.3799$ &    $-\infty$ & $-36.9537$ \\  
    \rowcolor{yellow} & $\mu_{385}$   & $-\infty$ & $-37.1365$ &   $-\infty$ & $-35.9399$  \\  
    \rowcolor{yellow} & $\mu_{449}$   & $-\infty$ & $-38.9673$ &    $-\infty$ &  $-44.9405$ \\ 
    \rowcolor{yellow} & $\mu_{32785}$ & $-\infty$ & $-40.1221$ &     $-\infty$ & $-45.8318$ \\ 
    \rowcolor{yellow} & $\mu_{389}$   & $-\infty$ & $-43.7297$ &     $-\infty$ &  $-40.0158$ \\
    \hline
    \rowcolor{cyan} $i\in I\setminus F_{2}'$ & $\mu_{256}$ &  $-\infty$ &  $-35.5482$    &  $-\infty$ &      $-\infty$ \\ 
    \rowcolor{cyan} & $\mu_{320}$ & $-\infty$ &  $-42.5454$    &  $-\infty$ &      $-\infty$ \\ 
    \rowcolor{cyan} & $\mu_{257}$ & $-\infty$ &  $-52.9224$    &  $-\infty$ &      $-\infty$ \\ 
    \rowcolor{cyan} & $\mu_{321}$ & $-\infty$ &  $-60.2208$   &  $-\infty$ &      $-\infty$ \\ 
    \hline
  \end{tabular}
  \label{table: mle}
\end{table}

For reference, we list the estimates of the top five cell counts obtained using our method and compare them with those
obtained by other methods in \citet{DobraLenkoski11:Copula_Gaussian_graphical_models} in
Appendix~\ref{S-sec:NLTCS-frequencies}.

\section{Computing faces for large complexes}
\label{sec:large-graphs}

If our statistical model contains many variables and is not reducible,
the problem of determining $\Fbf_{t}$ quickly becomes infeasible.  Not only does the marginal polytope become very
complicated, but also the size of the objects that one has to store or compute grows exponentially.
Consider for example a $10\times10$ grid of binary random variables.  This hierarchical model
has~280 parameters, 
and the total sample space has cardinality $|I|=2^{100}\approx\num{1.27E30}$.
If $\face_{t}$ is close to~$I$, we cannot even list the elements of~$\face_{t}$, which consists of approximately
$10^{30}$ elements.  Therefore, we take a local approach and look for separators.

If the simplicial complex $\Delta$ contains a complete separator separating $V$ into $V_1$ and $V_2$, we can identify a
facial set $F$ implicitly without listing it explicitly.  We only need the two projections $F_{V_{1}}=\pi_{V_{1}}(F)$
and $F_{V_{2}}=\pi_{V_{2}}(F)$.  Since $F=\pi_{V_{1}}^{-1}(F_{V_{1}})\cap\pi_{V_{2}}^{-1}(F_{V_{2}})$ (by
Lemma~\ref{lem:reducible-facial}), these two projections identify~$F$, and they allow us to do most of the operations
that we would want to do with~$F$.  For example, for any $i\in I$, we can check whether $i\in F$ by checking whether
$\pi_{V_{1}}(i)\in F_{V_{1}}$ and $\pi_{V_{2}}(i)\in F_{V_{2}}$, and we can check whether $F\supseteq I$ by checking whether
$F_{V_{1}}\supseteq I_{V_{1}}$ and~$F_{V_{2}}\supseteq I_{V_{2}}$.  In particular, we can check whether the MLE exists by looking
only at the two subsets $V_{1}$ and~$V_{2}$.

Similar ideas apply if $\Delta$ contains a separator that is not complete.
Suppose that $S$ separates $V_{1}$ from $V_{2}$ in~$\Delta$.  We want to use
$F_{2}:=\face_{\Delta|_{V_{1}}}(I_{+})\cap\face_{\Delta|_{V_{2}}}(I_{+})$ as an outer approximation and
$F_{1}:=\face_{\Delta_{S}}(I_{+})$ as an inner approximation to~$F_{t}$.  Due to the
problems mentioned above, we do not directly compute $F_1$ and~$F_2$, but we compute their projections on $V_1$ and~$V_2$.
Instead of~$F_{2}$, we compute the facial set $F_{2,V_{1}}:=\face_{\Delta|_{V_{1}}}(\pi_{V_{1}}(I_{+}))$ of the
$V_{1}$-marginal $\pi_{V_{1}}(I_{+})$ with respect to~$\Delta|_{V_{1}}$.  Similarly we
compute~$F_{2,V_{2}}:=\face_{\Delta|_{V_{2}}}(\pi_{V_{2}}(I_{+}))$.  Instead of~$F_{1}$, we compute
$F_{1,V_{1}}:=\face_{\Delta_{S}|_{V_{1}}}(\pi_{V_{1}}(I_{+}))$
and~$F_{1,V_{2}}:=\face_{\Delta_{S}|_{V_{2}}}(\pi_{V_{2}}(I_{+}))$.  Then we could recover $F_1$ and $F_2$ from the
equations
\begin{equation*}
  F_{2} = \pi_{V_{1}}^{-1}(F_{2,V_{1}})\cap\pi_{V_{2}}^{-1}(F_{2,V_{2}})
  \quad\text{ and }\quad
  F_{1} = \pi_{V_{1}}^{-1}(F_{1,V_{1}})\cap\pi_{V_{2}}^{-1}(F_{1,V_{2}}).
\end{equation*}
For any $x\in I$, we can check whether $x\in F_{1}$ by checking whether $\pi_{V_{1}}(x)\in F_{1,V_{1}}$ and
$\pi_{V_{2}}(x)\in F_{1,V_{2}}$.  More importantly, we can check whether $F_{1}=F_{2}$ by checking whether
$F_{1,V_{1}}=F_{2,V_{1}}$ and~$F_{1,V_{2}}=F_{2,V_{2}}$.
This idea can be applied iteratively when either $\Delta|_{V_{1}}$ or~$\Delta|_{V_{2}}$ has a separator.

The next two subsections illustrate these ideas.  In Section \ref{sec:US-senate}, we consider a graph with no particular
regularity pattern on 50 nodes and identify convenient separators.
In Section \ref{sec:5x10}, we consider a grid graph and work with two families of ``parallel'' separators that can be
used to iteratively improve the inner approximation.

\subsection{US Senate Voting Records Data}
\label{sec:US-senate}

\begin{figure}[th]
  \centering
  \includegraphics[scale=0.8]{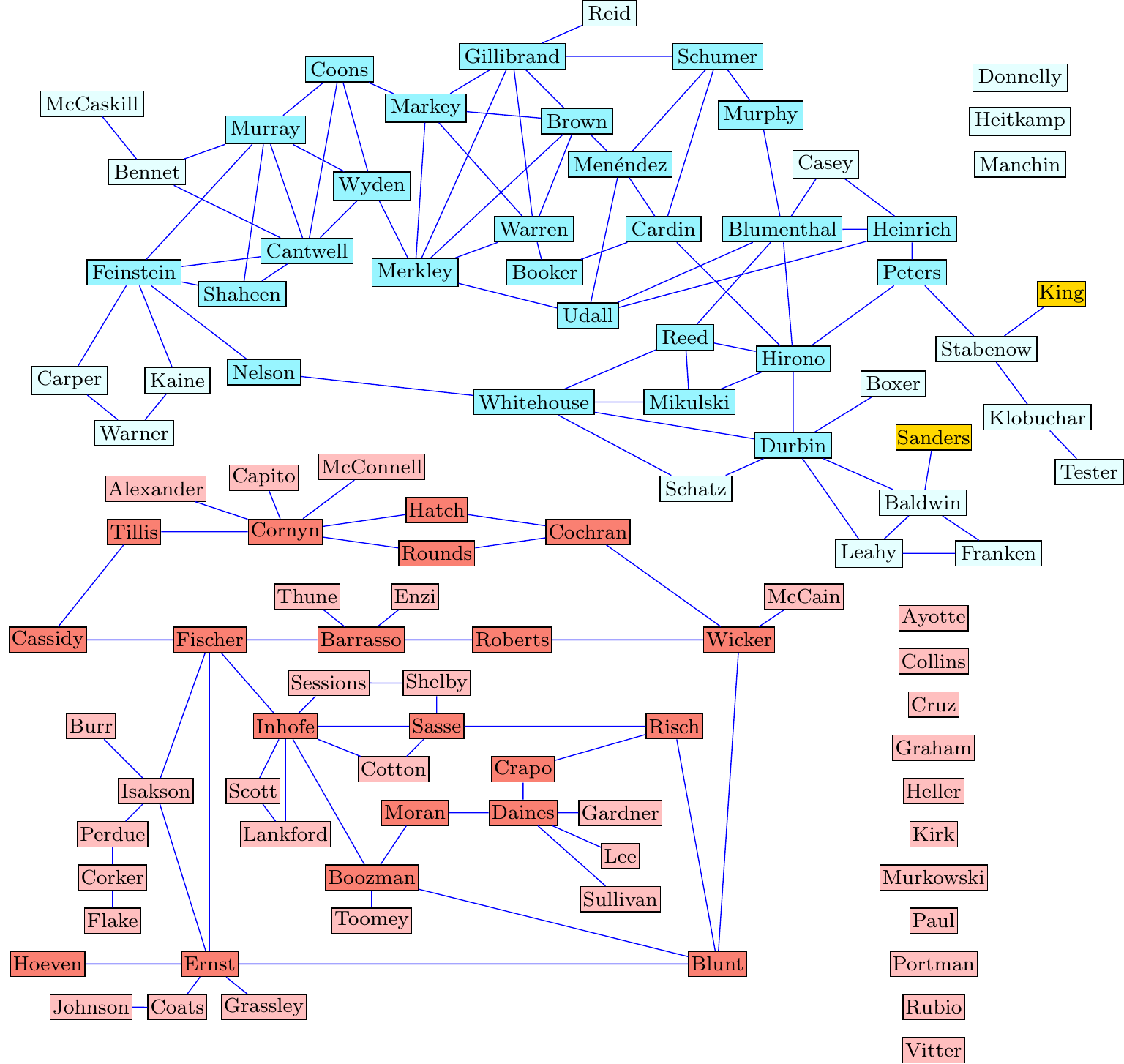}
  \caption{The graph for the US senate voting records data.  Golden nodes are independent senators, blue nodes are Democratic, and red nodes are Republican.}
  \label{voting}
\end{figure}

We consider the voting record of all 100 US senators on 309 bills from January 1 to November 19 2015.  Similar data for
the years 2004--2006 was analyzed by \citet{Banerjee08:model_selection}. The votes are recorded as ``yea,'' ``nay'' or
``not voting.''  We transformed the ``not voting'' into ``nay'' and consequently have a 100-dimensional binary data
set. To fit a hierarchical model to this data set, we use the $\ell_1$-regularized logistic regression method proposed
by \citet{Ravikumar10:logistic_regression} to identify the neighbours of each variable and construct an Ising model.
We set the regularization  parameter to $\lambda=32\sqrt{{\log p}/{n}} \approx 0.35$. 
The underlying graph of the Ising model is given in Figure~\ref{voting}.
This figure should not be interpreted as the graph of a graphical model.  Rather, the edges in the graph indicate where the two-way interactions lie.
There are 277 parameters in this model (the number of vertices plus the number of edges). The graph consists of two separate 
large connected components and 14 independent nodes.

There are 309 sample points, and $|I_+|=278$.  We want to characterize the face $\Fbf_{t}$ of the data on the marginal polytope.
The graph in Figure~\ref{voting} has many complete separators, and it decomposes as a union of several small
  irreducible simplicial subgraphs and two large irreducible subgraphs, one in each of the large connected components,
  as shown in Figure~\ref{fig:two-parties}.  By Lemma~\ref{lem:reducible-facial}, we can restrict attention to these
  irreducible subgraphs.  For the small irreducible subgraphs, one easily verifies that the data does not lie on a
  proper face of their corresponding marginal polytopes.  We are left with the two large irreducible prime components in
  Figure~\ref{fig:two-parties}.

\begin{figure}[thp] 
  \subfloat[\label{fig:voting-r}]{\includegraphics[ width=.5\linewidth]{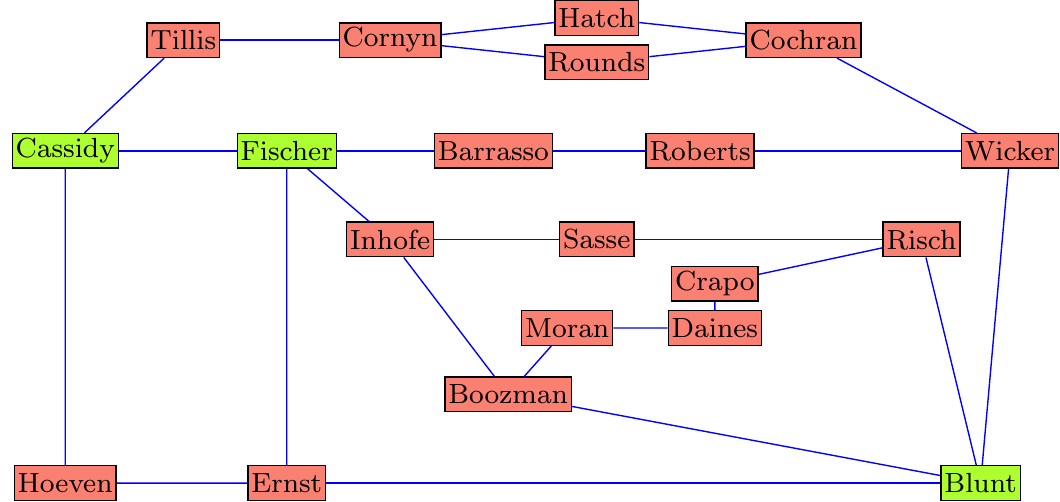}}\hfill
  \subfloat[\label{fig:voting-d}]{\includegraphics[ width=.5\linewidth]{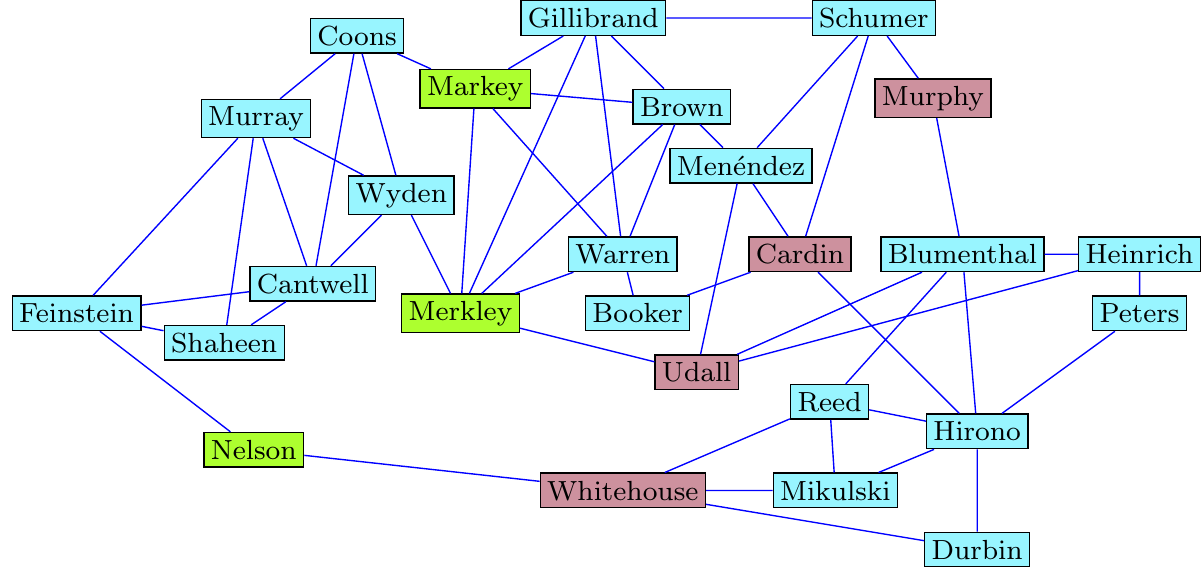}}
\caption{The simplicial complexes after cutting off the small prime components: \textup{(\protect\subref*{fig:voting-r})}~the Republican party prime component~$\Delta_r$.  \textup{(\protect\subref*{fig:voting-d})}~the Democratic party prime component~$\Delta_d$. The yellow and pink nodes are the two separator sets we used to compute the facial set.}
\label{fig:two-parties}
\end{figure}
 
The Democratic party simplicial complex $\Delta_d$ consists of 26 variables, and the model induced from $\Delta_d$
contains 77 parameters, which is too large to use linear programming to compute the face of
$\Pbf_{\Delta_{d}}$ containing the vector~$t_d$.
Therefore, we look for separators in order to obtain good inner and outer approximations.
Figure \ref{fig:voting-d} indicates (in yellow and pink) two separators that separate $\Delta_d$ into three simplicial complexes denoted, from left to right, by $\Delta_{\alpha}$, $\Delta_{\beta}$ and $\Delta_{\gamma}$ that are small enough for the linear programming method.

$\Delta_{\alpha}$ has 9 nodes. 
The corresponding vector $t_{\alpha}$ lies in the relative interior of~$\Pbf_{\Delta_{\alpha}}$.

\begin{table}
\centering
\begin{tabular}{cl@{\hspace{3\tabcolsep}}cl@{\hspace{3\tabcolsep}}cl@{\hspace{3\tabcolsep}}cl} 
 \toprule
 ID & Senator & ID & Senator & ID & Senator & ID & Senator \\
 \midrule
 22 & Nelson  &   37 & Cardin &   52 & Murphy    & 61 & Whitehouse \\
 23 & Reed    &   41 & Markey &   53 & Hirono    & 87 & Warren \\
 26 & Schumer &   47 & Udall  &   56 & Gillibrand  \\
 \bottomrule
\end{tabular}
\caption{Numbering of some senators}\label{tab:senator-ID}
\end{table}

$\Delta_{\beta}$ has 13 nodes, 
and $t_{\beta}$ lies on a facet $\Fbf_{t_{\beta}}$ of~$\Pbf_{\Delta_{\beta}}$.
To simplify the notation, we denote the 100 senators not by their name but by an integer between 1 and~100. We
only need to identify a few and their numbers are given in Table~\ref{tab:senator-ID}.
The inequality of $\Fbf_{t_{\beta}}$ is
\begin{equation}
  \label{warren}
  t_{87}-t_{56,87} \ge 0,
\end{equation}
where $t_{\text{87}}$ denotes the marginal count of senator Warren voting ``yea'' and $t_{\text{56},\text{87}}$ denotes the marginal counts of both senators Gillibrand and Warren voting ``yea.''

$\Delta_\gamma$ has 11 nodes.
The data vector $t_{\gamma}$ lies on the facet of $\Pbf_{\Delta_{\gamma}}$ with inequality
\begin{equation}
  \label{reed}
  t_{23}-t_{23,53} \ge 0.
\end{equation}
The intersection of the two facets \eqref{warren} and \eqref{reed} gives the outer aproximation $\Fbf_{2,d}$ to $F_{t_{d}}$.

To get an inner approximation, we complete each separator, i.e.\ the yellow vertices are completed and the pink vertices are completed in Figure~\ref{fig:voting-d}.  Denote the three simplicial complexes with complete separators as $\Delta_{\tilde{\alpha}},\ \Delta_{\tilde{\beta}},\ \Delta_{\tilde{\gamma}}$ respectively.  Then $\Delta_{\tilde{d}}=\Delta_{\tilde{\alpha}} \cup \Delta_{\tilde{\beta}} \cup \Delta_{\tilde{\gamma}}$ is a simplicial complex with two complete separators. The smallest face $\Fbf_{t_{\tilde{d}}}$ of the marginal polytope $\Pbf_{\Delta_{\tilde{d}}}$ containing the data vector $t_{\tilde{d}}$ is our inner approximation.  The models of $\Delta_{\tilde{\alpha}}$, $\Delta_{\tilde{\beta}}$, $\Delta_{\tilde{\gamma}}$ and $\Delta_{\tilde{d}}$ include main effects,  two-, three- and four-way interactions.  The dimension of the model induced by $\Delta_{\tilde{d}}$ is  91 (completing the two separators adds 14 parameters to the original model).  We apply the linear programming method to  $\Pbf_{\Delta_{\tilde{\alpha}}}$, $\Pbf_{\Delta_{\tilde{\beta}}}$ and $\Pbf_{\Delta_{\tilde{\gamma}}}$.

The dimension of the model of $\Delta_{\tilde{\alpha}}$ is 27,
and $\Fbf_{t_{\tilde{\alpha}}}$ is a facet with equation
\begin{equation} 
  \label{eq:face1}
  \<g_1,\  t_{\tilde{\alpha}} \> =t_{41}-t_{22,41}-t_{41,70}+t_{22,41,70} = 0.
\end{equation}
It follows that $\{g_1\}$ is a basis of the kernel of~$A_{F_{\tilde{\alpha}}}^{t}$.

The dimension of the model for $\Delta_{\tilde{\beta}}$ is 48.
The face $\Fbf_{t_{\tilde{\beta}}}$ has codimension~5, with defining equations
\begin{equation} 
  \label{eq:face2}
  \begin{cases}
    \<g_2,\  t_{\tilde{\beta}} \> =t_{87}-t_{56,87}=0 \\
    \<g_3,\  t_{\tilde{\beta}} \> = t_{47,52,61}+t_{37,52}-t_{37,52,61}-t_{37,47,52}=0\\
    \<g_4,\  t_{\tilde{\beta}} \> =t_{37, 47 ,52,61}-t_{ 47 ,52,61}=0  \\
    \<g_5,\  t_{\tilde{\beta}} \> =t_{37,52}+t_{26}-t_{26,52}-t_{26,37}=0 \\
    \<g_6,\  t_{\tilde{\beta}} \> =t_{41}-t_{22,41}-t_{41,70}+t_{22,41,70}=0 
  \end{cases}.
\end{equation}
Again, $\{g_2,g_3,g_4,g_5,g_6\}$ is a basis of the kernel of $A_{F_{\tilde{\beta}}}$.

The dimension of the model for $\Delta_{\tilde{\gamma}}$ is 38.
The face $\Fbf_{t_{\tilde{\gamma}}}$ has codimension~3.  It is defined by the equations
\begin{equation} 
  \label{eq:face3}
  \begin{cases}
    \<g_7,\  t_{\tilde{\gamma}} \> = t_{47,52,61}+t_{37,52}-t_{37,52,61}-t_{37,47,52}=0 \\
    \<g_8,\  t_{\tilde{\gamma}} \>  =t_{37, 47 ,52,61}-t_{ 47 ,52,61}=0 \\
    \<g_9,\  t_{\tilde{\gamma}} \> =t_{23}-t_{23,53}=0\;. 
  \end{cases}.
\end{equation}
Again, $\{g_7,g_8,g_9\}$ is a basis of the kernel of $A_{F_{\tilde{\gamma}}}$.

From Lemma~\ref{lem:reducible-facial},  $\Fbf_{t_{\tilde{d}}}=\Fbf_{\tilde{\alpha}} \cap 
\Fbf_{\tilde{\beta}} \cap \Fbf_{\tilde{\gamma}}$, and the equations for $\Fbf_{t_{\tilde{d}}}$ are
\begin{equation} 
  \label{eq:face4}
  \begin{cases}
    \<g'_1,\  t_{\tilde{d}} \> =t_{41}-t_{22,41}-t_{41,70}+t_{22,41,70}=0 \\
    \<g'_2,\  t_{\tilde{d}} \> =t_{87}-t_{56,87}=0 \\
    \<g'_3,\  t_{\tilde{d}} \> = t_{47,52,61}+t_{37,52}-t_{37,52,61}-t_{37,47,52}=0\\
    \<g'_4,\  t_{\tilde{d}} \> =t_{37, 47 ,52,61}-t_{ 47 ,52,61}=0  \\
    \<g'_5,\  t_{\tilde{d}} \> =t_{37,52}+t_{26}-t_{26,52}-t_{26,37}=0 \\
    \<g'_9,\  t_{\tilde{d}} \> =t_{23}-t_{23,53}=0,
  \end{cases}
\end{equation}
where the vectors $g'_1,\dots,g'_9$ are the vectors $g_{1},\dots,g_{9}$ extended to $\R^{91}$ by adding zeros on the
corresponding complementary coordinates (since $g'_1=g'_6$, $g'_3=g'_7$, $g'_4=g'_8$, only six of the
nine equations are needed).  Thus, $\Fbf_{1,d}:=\Fbf_{t_{\tilde d}}$, defined by~\eqref{eq:face4}, is a strict subset of the face~$\Fbf_{2, d}$ defined by \eqref{warren} and~\eqref{reed}.  Next, we 
refine our argument and show that indeed~$\Fbf_{t_{d}}=\Fbf_{2, d}$.

From what we know, it follows that the orthogonal complement of the subspace generated by $\Fbf_{t_{\tilde{d}}}$ is
\begin{equation*}
  G=\Big\{g' \in\R^{91}\Big|g'=k_1g'_1+k_2g'_2+k_3g'_3+k_4g'_4+k_5g'_5+k_9g'_9\Big\}.
\end{equation*}
To describe~$\Fbf_{t_{d}}$, we note that each defining equation of~$\Fbf_{t_{d}}$ is of the form $\<g,t_d\>=0$, where $g$ is orthogonal to~$\Fbf_{t_{d}}$.  For any such $g$, let $g'$ be its extension to a vector in $\R^{91}$ by adding zero components.  Then $g'\perp\Fbf_{t_{\tilde d}}$, which implies that $g'\in G$.
Therefore, we can find  $g$ by finding all vectors $g'\in G$ that vanish on all added components.  This yields a system of linear equations in~$k_{1},\dots,k_{5},k_{9}$.  We claim that all solutions must satisfy
$k_1=k_3=k_4=k_5=0$. Indeed, the coefficient of any triple or quadruple interaction must vanish (since these do not
belong to the original Ising model), which implies $k_1=k_3=k_4=0$, and also the coefficient of $t_{37,52}$ must vanish,
which implies~$k_{5}=0$.  On the other hand, the vectors $g'_{2}$ and~$g'_{9}$ only contain interactions that are
already present in~$\Delta$, and so the coefficients $k_{2}$ and~$k_{9}$ are free.  Thus the equations for $\Fbf_{t_d}$
are
\begin{equation} 
  \label{eq:face5}
  \begin{cases}
    \<g_2,\  t_{\tilde{\beta}} \> =t_{87}-t_{56,87}=0, \\
    \<g_9,\  t_{\tilde{\gamma}} \> =t_{23}-t_{23,53}=0.
  \end{cases}
\end{equation}
This is the same as the outer approximation $\Fbf_{2,d}$.

The Republican simplicial complex $\Delta_r$ consists of 20 variables, and the model induced from $\Delta_r$ contains 46
parameters, 
which is also too large to directly
compute~$F_{t_{r}}$.  The yellow nodes in Figure~\ref{fig:voting-r} separate $\Delta_r$ into two simplicial complexes
denoted (from left to right) by $\Delta_a$ and $\Delta_b$.
To compute the inner approximation, we complete the yellow separator and obtain two new simplicial complexes
$\Delta_{\tilde{a}}$ and $\Delta_{\tilde{b}}$.  With linear programming, we find that the corresponding
vectors $t_{\tilde{a}}$ and $t_{\tilde{b}}$ lie in the relative interior of the polytopes $\Pbf_{\Delta_{\tilde{a}}}$ and
$\Pbf_{\Delta_{\tilde{b}}}$, respectively.  Therefore, $\Fbf_1=\Pbf_{\Delta_r}$, from which
we conclude that the corresponding vector $t_r$ lies in the relative interior of~$\Pbf_{\Delta_r}$.

Thus, the face $\Fbf_{t}$ is now determined: it is characterized by the equalities~\eqref{eq:face5}.
  What insight is there in the knowledge (i) of the non-existence of the MLE and (ii) of the exact
  face $\Fbf_{t}$?  While we have given general remarks in the introduction, let us illustrate here  how knowledge about $\Fbf_{t}$ points to some issues with the statistical analysis that would
  possibly be overlooked if $\Fbf_{t}$ was not known.
  
First, knowing $\Fbf_{t}$, and its defining inequalities, for one model also gives information about other models.
It follows from~\eqref{eq:face5} that the MLE does not exist for any hierarchical model that includes one of the edges
$(23,53)$ or $(56,87)$ (to see this, note that the inequality~\eqref{reed} 
defines a proper face for any model containing the edge~$(23,53)$, since the corresponding sufficient statistics vector
satisfies the equality in ~\eqref{reed}).
Thus, if one wants to find a smaller model,  within
  the realm of hierarchical models, for which the MLE exists, both edges have to be dropped. However, from the data, here, evidence for both
  edges is quite strong, and thus the edges should not be dropped.

Second, let us consider the computation of the EMLE.
  As we know $\Fbf_{t}$, instead of running
an MLE computation for a model with 277 parameters and $2^{100}$ outcomes, we are left with an MLE computation for a model  with $277-2=275$ parameters and
 $|F_{t}| = \frac{9}{16}\cdot 2^{100}$ outcomes, those without the configurations $(X_{23}, X_{53})=(1,0)$ or $(X_{56}, X_{87})=(1,0)$ that all have counts zero.  These numbers are still too large for a direct
computation, even
when taking into account that the EMLE can be computed by restricting to each of the irreducible components. 
So, we turn to an approximate method and compute the  maximum composite likelihood estimate. The maximum composite likelihood estimate  can be obtained by combining estimates from the local conditional likelihoods derived from the distribution of each variable given its neighbours: see for example \cite{liu-ihler2012} or \cite{MassamWang18:Local_parameter_estimation}. Thus the reliability of the maximum composite likelihood estimates depends upon the existence of the maximum in each of the local conditional likelihood. These local conditional likelihoods are derived from the global model built on the entire cell set $I$ and, certainly in practice, without worrying about the existence of the global MLE. Let us consider, for example, the  likelihood obtained from the conditional distribution of $X_{23}$ given its neighbours $X_{19,53,61,78}$.  For convenience, let  $19,23, 53,61,78$ be denoted as $a, b, c, d, e$. This likelihood is the product over all configurations
of $i_{acde}$ in the data set of conditional binomial distributions for the variable $X_b$ and can be written  as
  \begin{equation}
  \label{lcp}
  \prod_{i_{a},i_{b},i_{c},i_{d},i_{e}} p\big(X_{b}=i_{b}\big| X_{acde} = (i_{a},i_{c},i_{d},i_{e})\big)^{n(i_{a},i_{b},i_{c},i_{d},i_{e})},
\end{equation}
where $n(i_{a},i_{b},i_{c},i_{d},i_{e})$ denotes the corresponding marginal cell count.
It is easy to show that the MLE of each $p(X_b=1\mid i_a,i_c, i_{d}, i_{e})$
is the empirical estimate $n(i_a, i_b=1, i_c, i_{d}, i_{e})/n(i_a, i_c,i_{d}, i_{e})$.
In the data set, $n(i_{a},i_b=1,i_c=0, i_{d}, i_{e})=0$ for all $i_{a},i_{d},i_{e}\in\{0,1\}$.
Thus,
$$\hat{p}(X_b=1\mid i_a=1, i_c=0, i_d=1, i_e=1)=\frac{\exp(\hat{\theta}_b+\hat{\theta}_{ab}+\hat{\theta}_{bd}+\hat{\theta}_{be})}{1+\exp(\hat{\theta_b}+\hat{\theta}_{ab}+\hat{\theta}_{bd}+\hat{\theta}_{be})}=0,$$
so that  $\hat{\theta}_b+\hat{\theta}_{ab}+\hat{\theta}_{bd}+\hat{\theta}_{be}=-\infty$, and the MLE of at least some of these parameters, which are the corresponding parameters of the global model,  does not exist.
Now the composite maximum likelihood estimate is obtained by averaging the estimates obtained from various local conditional likelihoods. From the $b$-local conditional model, we also obtain $\hat{p}(X_b=1\mid i_a=1,  i_c=1, i_d=1, i_e=0)=1/2$ and $\hat{p}(X_b=1\mid i_a=0, i_c=1, i_d=0, i_e=1)=4/5$, which yield the linear combinations
\begin{equation}
\label{combi}
\hat{\theta}_b+\hat{\theta}_{bc}+\hat{\theta}_{ab}+\hat{\theta}_{bd}=0,\;\;\hat{\theta}_b+\hat{\theta}_{bc}+\hat{\theta}_{be}=1.4.
\end{equation}

The remarks above are verified numerically. Let $\theta=(\theta_b, \theta_{ab}, \theta_{bc},\theta_{bd},\theta_{be})$, and denote by $l_{\text{local}}(\theta)$ the local conditional likelihood.
  Starting at $\theta_0=(0,0,0,0,0)$  and optimizing \eqref{lcp} in terms of $\theta$ in Matlab, we obtain $l_{\text{local}}(\hat{\theta})=3.88830675$ 
  for the maximum 
  $\hat{\theta}=\hat{\theta}_1\approx(-62.3 ,\  -16.8 ,\    35.2,\   43.9 ,\   28.5)$.  If we change the starting point to $\theta_0=(100,100,100,100,100)$, we obtain $\hat{\theta}_2=(-162.2, -26.1, 91.2,  97.1,  72.4)$ and $l_{\text{local}}(\hat{\theta}_2)=3.88830648$. 
  Clearly, the values for $\hat{\theta}$ are unreliable since the MLE in the local conditional model does not exist. However, both $\hat{\theta}_1$ and $\hat{\theta}_2$ satisfy equations \eqref{combi}.
  One can, of course, obtain estimates of $\theta_{bc}, \theta_{ab}, \theta_{bd},\theta_{be}$ from the local conditional models centered at $c,a,d$ and $e$ respectively but these estimates will not have been obtained through the method of composite likelihood and it remains to study their properties.

This example shows that our method scales well and makes it possible to obtain the face $\Fbf_t$ for very large examples.  It also illustrates how knowing the face $\Fbf_t$ gives us precious information on the reliability of the maximum composite likelihood estimate.

\subsection{The \texorpdfstring{$5\times10$}{5x10}-grid}
\label{sec:5x10}

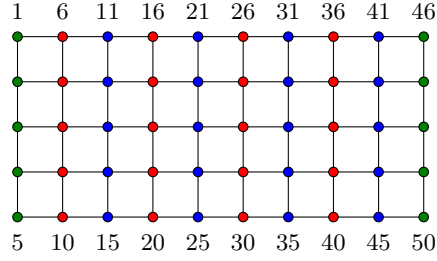
\begin{figure} 
  \centering
  \begin{tikzpicture}[scale=0.6]
    \griddefinenodes{5}{10}
    \griddrawedges{5}{10}
    \foreach \j in {1,10} \foreach \i in {1,...,5} {
      \filldraw[fill=green!50!black] (X\i\j) circle (3pt); }
    \foreach \j in {2,4,...,8} \foreach \i in {1,...,5} {
      \filldraw[fill=red] (X\i\j) circle (3pt);  }
    \foreach \j in {3,5,...,9} \foreach \i in {1,...,5} {
      \filldraw[fill=blue] (X\i\j) circle (3pt); }
    \gridaddlabels{5}{10}
  \end{tikzpicture}
  \caption{$5 \times 10$ grid graph, the red and blue nodes are the set of separators we use to compute~$F_1$, they are used iteratively to get a better lower approximation}
  \label{5by10}
\end{figure}

Let $\Delta$ be the simplicial complex of the $5\times 10$ grid graph.  We exploit the regularity of this
graph and make use of the vertical separators in the grid to obtain inner and outer approximations of the facial
sets.  The graph has 50 nodes, which is too many to directly compute a facial set or even to store it.  However, the
$5\times10$ grid has 8 vertical separators marked in red and blue in Figure~\ref{5by10}, and we can use these to
approximate~$F_t$.  Since facial sets for $5\times 3$ grids can be computed reasonably fast (3 to 4 seconds on a laptop
with \SI{2.50}{GHz} processor 
and \SI{12}{GB} memory), we only use three of these vertical separators at a time, say the blue separators
\begin{equation*}
  S_{2}=\{11,\dots,15\},\;S_{4}=\{21,\dots,25\},\;S_{6}=\{31,\dots,35\},\;S_{8}=\{41,\dots,45\}.
\end{equation*}
These  separate the vertex sets
\begin{gather*}
  V_{1}=\{1,\dots,15\},\; V_{3}=\{11,\dots,25\},\; V_{5}=\{21,\dots,35\},\\
  V_{7}=\{31,\dots,45\},\; V_{9}=\{41,\dots,50\}.
\end{gather*}

\begin{figure} 
  \centering
  \begin{tikzpicture}[scale=0.6]
    \foreach \i in {0,...,3} {
      \begin{scope}[shift={(3*\i,0)}]
        \labelledgrid[10*\i+1]{5}{3}
      \end{scope}}
    \begin{scope}[shift={(12,0)}]
      \labelledgrid[41]{5}{2}
    \end{scope}
  \end{tikzpicture}

  \caption{Five induced sub-grids}
  \label{fig:local5}
\end{figure}
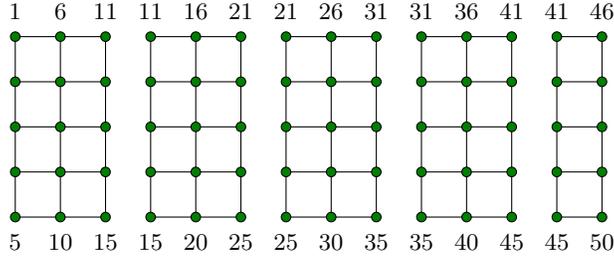

Adding the blue separators to $\Delta$ gives a simplicial complex
\begin{equation*}
\Delta_{S_{2};S_{4};S_{6};S_{8}}:=\Delta\bigcup_{j=2,4,6,8}\{F:F\subseteq S_{j}\}
\end{equation*}
with five irreducible components supported on the vertex sets $V_{1},V_{3},V_{5},V_{7}$ and $V_{9}$
(Figure~\ref{fig:5x10+seps-both}).  To compute a facial set with respect to~$\Delta_{S_{2};S_{4};S_{6};S_{8}}$,
according to Lemma~\ref{lem:reducible-facial} applied four times, we need to compute
\begin{gather*}
  G_{1,V_{1}}:=\face_{\Delta_{S_{2}}|_{V_{1}}}(\pi_{V_{1}}(I_{+})), \quad
  G_{1,V_{3}}:=\face_{\Delta_{S_{2};S_{4}}|_{V_{3}}}(\pi_{V_{3}}(I_{+})), \\
  G_{1,V_{5}}:=\face_{\Delta_{S_{4};S_{6}}|_{V_{5}}}(\pi_{V_{5}}(I_{+})), \quad
  G_{1,V_{7}}:=\face_{\Delta_{S_{6};S_{8}}|_{V_{7}}}(\pi_{V_{7}}(I_{+})), \\
  G_{1,V_{9}}:=\face_{\Delta_{S_{8}}|_{V_{9}}}(\pi_{V_{9}}(I_{+})).
\end{gather*}
Then $G_{1}:=\bigcap_{i}\pi_{V_{i}}^{-1}(G_{1,V_{i}})$ is equal to~$\face_{\Delta_{S_{2};S_{4};S_{6};S_{8}}}(I_{+})$,
and thus an inner approximation of~$F_{t}$. As stated before, we do not need to compute $G_{1}$ explicitly, but we represent it by
means of the~$G_{1,V_{i}}$.  

We can improve the approximations by also considering the red separators 
\begin{equation*}
  S_{1}=\{6,\dots,10\},\;S_{3}=\{16,\dots,20\},\;S_{5}=\{26,\dots,30\},\;S_{7}=\{36,\dots,40\},
\end{equation*}
that separate
\begin{gather*}
  V_{0}=\{1,\dots,10\},\; V_{2}=\{6,\dots,20\},\; V_{4}=\{16,\dots,30\}, \\ 
  V_{6}=\{26,\dots,40\},\; V_{8}=\{36,\dots,50\}.
\end{gather*}
As explained in Section~\ref{sec:inner-approximations},
we want to compute $G^{(2)}_{1} := \face_{\Delta_{S_{1};S_{3};S_{5};S_{7}}}(G_{1})$.  Again, instead of computing $G^{(2)}_{1}$
directly, we need only compute  the much smaller sets
$G^{(2)}_{1,V_{0}}:=\pi_{V_{0}}(G^{(2)}_{1}),G^{(2)}_{1,V_{2}}:=\pi_{V_{2}}(G^{(2)}_{1}),\dots,G^{(2)}_{1,V_{8}}:=\pi_{V_{8}}(G^{(2)}_{1})$.
So the question is: is it possible to compute $G^{(2)}_{1,V_{0}}$, $G^{(2)}_{1,V_{2}}$, \dots, $G^{(2)}_{1,V_{8}}$ from
$G_{1,V_{1}},G_{1,V_{3}},\dots,G_{1,V_{9}}$, without computing $G_{1}$ in between?

It turns out that this is indeed possible: By Lemma~\ref{lem:reducible-facial}, all we need to compute~$G^{(2)}_{1,V_{i}}$
is~$G_{1,V_{j}}:=\pi_{V_{j}}(G_{1}),\;j=i-1,i+1$.  For $i=0$, since $V_{0}\subset V_{1}$, we can compute $G_{1,V_{0}}$ from
$\pi_{V_{1}}(G_{1})=G_{1,V_{1}}$.  For $i=2,4,6,8$, since $V_{i}\subset V_{i-1}\cup V_{i+1}$, we can compute
$G_{1,V_{i}}$ from~$\pi_{V_{i-1}\cup V_{i+1}}(G_{1})$, which itself can be obtained by ``gluing''
$\pi_{V_{i-1}}(G_{1})=G_{1,V_{i-1}}$ and $\pi_{V_{i+1}}(G_{1})=G_{1,V_{i+1}}$:
\begin{equation*}
  \pi_{V_{i-1}\cup V_{i+1}}(G_{1}) = \left(\pi^{V_{i-1}\cup V_{i+1}}_{V_{i-1}}\right)^{-1}(G_{1,V_{i-1}}) \cap \left(\pi^{V_{i-1}\cup V_{i+1}}_{V_{i+1}}\right)^{-1}(G_{1,V_{i+1}}),
\end{equation*}
where $\pi^{V'}_{V''}$ for $V''\subseteq V'$ denotes the marginalization map from $I_{V'}$ to~$I_{V''}$ and where $\Big(\pi^{V'}_{V''}\Big)^{-1}$ denotes the lifting from $I_{V''}$ to $I_{V'}$.

As explained in Section~\ref{sec:inner-approximations}, we have to iterate this procedure: From $G^{(2)}_{1}$ we want to
compute $G^{(3)}_{1}:=\face_{\Delta_{S_{2};S_{4};S_{6};S_{8}}}(G_{1}')$ or, more precisely, we want to compute
$G^{(3)}_{1,V_{i}}=\pi_{V_{i}}(G^{(3)}_{1})$ for~$i=1,3,\dots,9$.  Again, we do this without looking at~$G^{(2)}_{1}$
directly by just using the information available through the~$G^{(3)}_{1,V_{i}}$.  Iterating this procedure, we obtain a
sequence of sets $G^{(k)}_{1,V_{i}}, G^{(k)}_{1,V_{j}}$ (with odd $i$ and even~$j$), which stabilizes after a finite
number of steps.  Let
\begin{equation*}
  F_{1,V_{i}} := 
    \bigcup G^{(k)}_{1,V_{i}},
\end{equation*}
Our best inner approximation is then $F_{1} = \bigcap_{i=0}^{9}\pi_{V_{i}}^{-1}(F_{1, V_i})$.  Again, we do not compute
$F_{1}$ explicitly, but we represent it in terms of the~$F_{1,V_{i}}$.
The process is visualized in Figure~\ref{fig:flow-chart}.

\newcommand{\griddeformtoleft}[1]{  
  \foreach \i in {2,4} {
    \path (#1-0.33,-\i) coordinate (X\i#1);}
  \path (#1-0.5,-3) coordinate (X3#1);}
\newcommand{\griddeformtoright}[1]{  
  \foreach \i in {2,4} {
    \path (#1+0.33,-\i) coordinate (X\i#1);}
  \path (#1+0.5,-3) coordinate (X3#1);}
\newcommand{\gridfillseparator}[2][fill=lightgray]{  
      \fill[#1] (X1#2) -- (X2#2) -- (X3#2) -- (X4#2) -- (X5#2) -- cycle;
      \draw (X1#2) -- (X3#2) -- (X5#2) -- (X2#2) -- (X4#2) -- (X1#2) -- (X5#2);}
\newcommand{\subgridwseparatorsleft}[3][fill=lightgray]{ 
  \griddefinenodes{5}{#2}
  \foreach \j in {#3} {
    \griddeformtoleft{\j}
    \gridfillseparator[#1]{\j}
  }
  \griddrawedges{5}{#2}
  \griddrawnodes{5}{#2}}
\newcommand{\subgridwseparatorsright}[3][fill=lightgray]{ 
  \griddefinenodes{5}{#2}
  \foreach \j in {#3} {
    \griddeformtoright{\j}
    \gridfillseparator[#1]{\j}
  }
  \griddrawedges{5}{#2}
  \griddrawnodes{5}{#2}}
\newcommand{\subgridseparatorlr}[1][fill=lightgray]{ %
  \foreach \i in {1,5} \foreach \j in {1,3} {
    \path (\j,-\i) coordinate (X\i\j);}
  \griddeformtoright{1}
  \foreach \i in {1,...,5} {
    \path (2,-\i) coordinate (X\i2);}
  \griddeformtoleft{3}
  \gridfillseparator[#1]{1}
  \gridfillseparator[#1]{3}
  \griddrawedges{5}{3}
  \griddrawnodes{5}{3}
}

\begin{figure}[ht] 
  \centering
  \subfloat[\label{fig:5x10+sep}]{%
  \begin{tikzpicture}[scale=0.6,font=\small]
    \subgridwseparatorsleft[fill=blue]{10}{3,5,7,9}
    \gridaddlabels{5}{10}
  \end{tikzpicture}}
  \hfil
  \subfloat[\label{fig:5x10-subgraphs+sep}]{%
  \begin{tikzpicture}[scale=0.6,font=\small]
    \subgridwseparatorsleft[fill=blue]{3}{3}
    \gridaddlabels{5}{3}
    \foreach \i in {1,...,3} {
      \begin{scope}[shift={(3*\i,0)}]
        \subgridseparatorlr[fill=blue]
        \gridaddlabels[10*\i+1]{5}{3}
      \end{scope}}
    \begin{scope}[shift={(12,0)}]
      \subgridwseparatorsright[fill=blue]{2}{1}
      \gridaddlabels[41]{5}{2}
    \end{scope}
  \end{tikzpicture}}

  \caption{\textup{(\protect\subref*{fig:5x10+sep})}~The $5\times10$-grid with the blue separators completed.  \textup{(\protect\subref*{fig:5x10-subgraphs+sep})}~The five irreducible subcomplexes after completing the separators.}
  \label{fig:5x10+seps-both}
\end{figure}
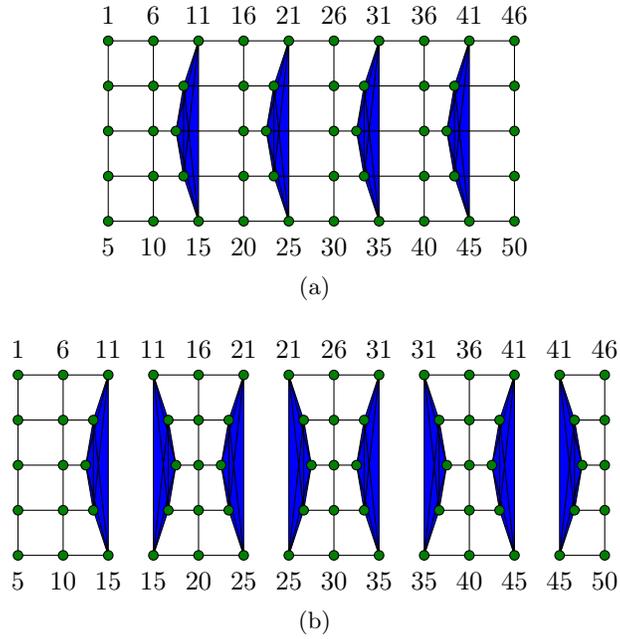

Let us now consider the outer approximation $F_2$.  We adapt Strategy~\ref{strat:all-subgrids} of Section~\ref{sec:outer-approximations} and cover the graph with $5\times 3$ grid subgraphs, since the facial sets for such graphs can easily be computed.  
These subgrids are supported on the same vertex subsets $V_i, i=1,\ldots,8$ as used when computing~$F_{1}$.  This makes
it possible to compare $F_{1}$ and~$F_{2}$.
 For $i=1,3,\dots,8$ we compute
$F_{2,V_{i}}=\face_{\Delta|_{V_{i}}}(\pi_{V_{i}}(I_{+}))$.  The outer approximation is then
$F_{2}=\bigcap_{i}\pi_{V_{i}}^{-1}(F_{2,V_i})$. Again, we don't compute $F_{2}$ explicitly, but
we only store $F_{2,V_{i}}$ in a computer as a representation of~$F_{2}$.
To compare the two approximations $F_1$ and $F_2$, we need only compare their projections $F_{1,V_i}$ and $F_{2,V_i}$ pairwise, $i=1,\ldots, 8$.

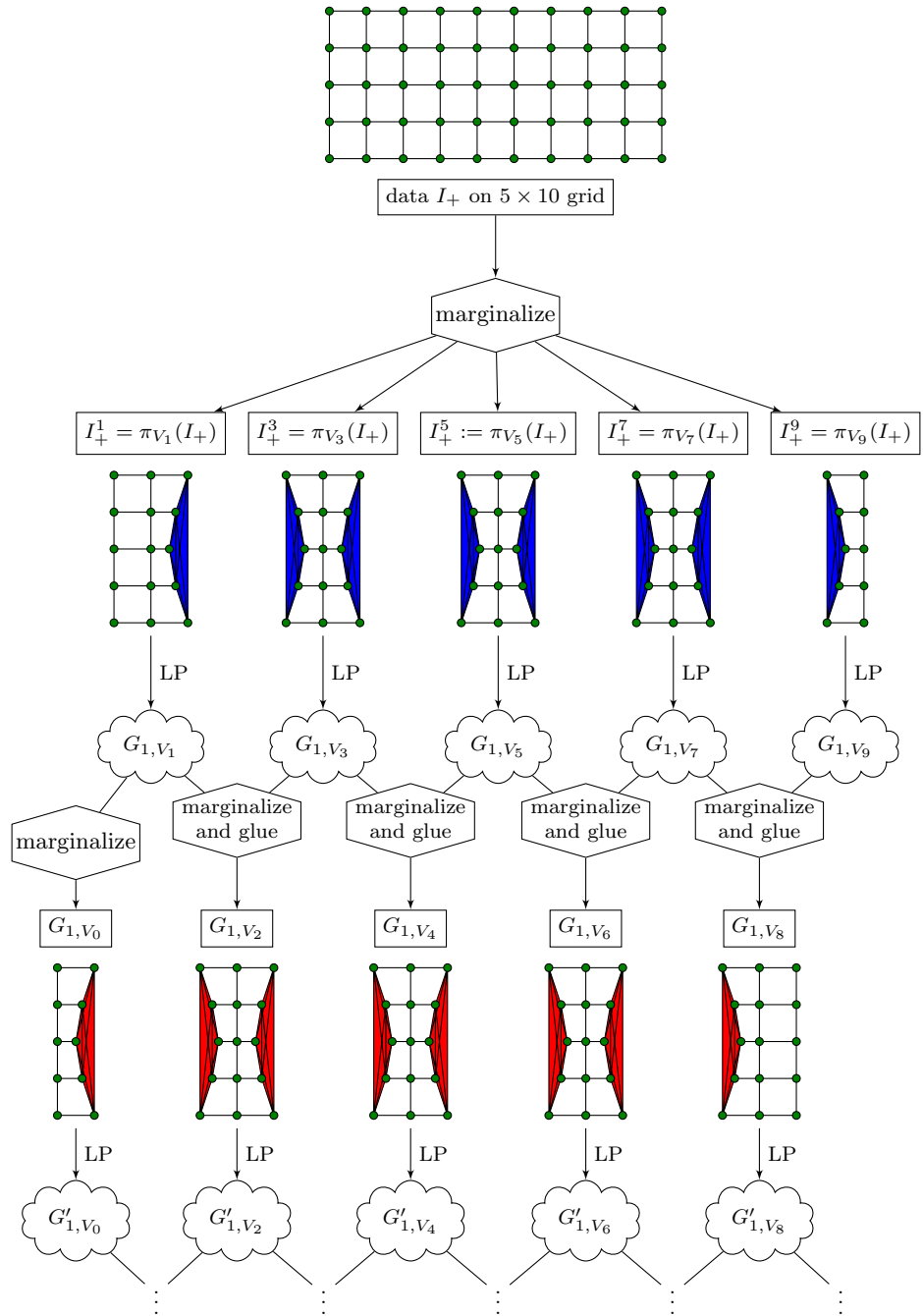
\begin{figure}
  \centering
  \tikzstyle{block} = [rectangle, draw]
  \tikzstyle{cloudy} = [cloud,cloud ignores aspect, draw]
  \tikzstyle{decision} = [regular polygon,regular polygon sides=6,xscale=2,shape border rotate=30,draw,inner sep=1pt]
  \begin{tikzpicture}[node distance=1.6cm,font=\scriptsize]
    \node [block] (start) {data $I_{+}$ on $5\times10$ grid};
    \node [above=0.1 of start] {\tikz[scale=0.5]{\grid{5}{10}}};
    \node [decision, below of=start,minimum width=1cm,label=center:{\footnotesize marginalize}] (marginalization) {}; 
    \node [block, below of=marginalization, xshift=1] (V5) {$I_{+}^{5}:=\pi_{V_{5}}(I_{+})$};
    \node [block, left=0.3 of V5] (V3) {$I_{+}^{3}=\pi_{V_{3}}(I_{+})$};
    \node [block, left=0.3 of V3] (V1) {$I_{+}^{1}=\pi_{V_{1}}(I_{+})$};
    \node [block,right=0.3 of V5] (V7) {$I_{+}^{7}=\pi_{V_{7}}(I_{+})$};
    \node [block,right=0.3 of V7] (V9) {$I_{+}^{9}=\pi_{V_{9}}(I_{+})$};
    \node [below=0.1 of V1] (g1) {\tikz[scale=0.5]{\subgridwseparatorsleft[fill=blue]{3}{3}}};
    \node [below=0.1 of V3] (g3) {\tikz[scale=0.5]{\subgridseparatorlr[fill=blue]}};
    \node [below=0.1 of V5] (g5) {\tikz[scale=0.5]{\subgridseparatorlr[fill=blue]}};
    \node [below=0.1 of V7] (g7) {\tikz[scale=0.5]{\subgridseparatorlr[fill=blue]}};
    \node [below=0.1 of V9] (g9) {\tikz[scale=0.5]{\subgridwseparatorsright[fill=blue]{2}{1}}};
    \node [cloudy,below=1 of g1] (f1) {$G_{1,V_{1}}$};
    \node [cloudy,below=1 of g3] (f3) {$G_{1,V_{3}}$};
    \node [cloudy,below=1 of g5] (f5) {$G_{1,V_{5}}$};
    \node [cloudy,below=1 of g7] (f7) {$G_{1,V_{7}}$};
    \node [cloudy,below=1 of g9] (f9) {$G_{1,V_{9}}$};
    \foreach \i/\j in {1/3,3/5,5/7,7/9} {
      \node [decision,minimum width=1cm,label=center:{\scriptsize\begin{tabular}{c}marginalize\\and glue\end{tabular}}] (M\i\j) at ($(f\i)!0.5!(f\j)-(0,1)$) {};}
    \node [block,below=0.7 of M13] (V2) {$G_{1,V_{2}}$};
    \node [block,below=0.7 of M35] (V4) {$G_{1,V_{4}}$};
    \node [block,below=0.7 of M57] (V6) {$G_{1,V_{6}}$};
    \node [block,below=0.7 of M79] (V8) {$G_{1,V_{8}}$};
    \node [block, left=1.2 of V2] (V0) {$G_{1,V_{0}}$};
    \node [decision,above=0.4 of V0,minimum width=1cm,label=center:{\footnotesize marginalize}] (M1) {}; 
    \node [below=0.1 of V0] (g0) {\tikz[scale=0.5]{\subgridwseparatorsleft[fill=red]{2}{2}}};
    \node [below=0.1 of V2] (g2) {\tikz[scale=0.5]{\subgridseparatorlr[fill=red]}};
    \node [below=0.1 of V4] (g4) {\tikz[scale=0.5]{\subgridseparatorlr[fill=red]}};
    \node [below=0.1 of V6] (g6) {\tikz[scale=0.5]{\subgridseparatorlr[fill=red]}};
    \node [below=0.1 of V8] (g8) {\tikz[scale=0.5]{\subgridwseparatorsright[fill=red]{3}{1}}};
    \node [cloudy,below=0.7 of g0] (f0) {$G'_{1,V_{0}}$};
    \node [cloudy,below=0.7 of g2] (f2) {$G'_{1,V_{2}}$};
    \node [cloudy,below=0.7 of g4] (f4) {$G'_{1,V_{4}}$};
    \node [cloudy,below=0.7 of g6] (f6) {$G'_{1,V_{6}}$};
    \node [cloudy,below=0.7 of g8] (f8) {$G'_{1,V_{8}}$};
    \node (M02) at ($(f0)!0.5!(f2)-(0,1)$) {$\vdots$};
    \node (M24) at ($(f2)!0.5!(f4)-(0,1)$) {$\vdots$};
    \node (M46) at ($(f4)!0.5!(f6)-(0,1)$) {$\vdots$};
    \node (M68) at ($(f6)!0.5!(f8)-(0,1)$) {$\vdots$};
    \node (M8)  at ($(f8)+(1.1,-1)$) {$\vdots$};
    \draw (f3) -- (M13);
    \draw (f5) -- (M35);
    \draw (f7) -- (M57);
    \draw (f9) -- (M79);
    \draw (f0) -- (M02) -- (f2) -- (M24) -- (f4) -- (M46) -- (f6) -- (M68) -- (f8) -- (M8);
    \begin{scope}[-latex']
      \draw (start) -- (marginalization);
      \foreach \i in {1,3,...,9} {
        \draw (marginalization) edge (V\i);}
      \foreach \i in {0,1,...,9} {
         \draw (g\i) -- (f\i) node[midway,right] {LP};}
      \draw (f1) -- (M1)  -- (V0);
      \draw (f1) -- (M13) -- (V2);
      \draw (f3) -- (M35) -- (V4);
      \draw (f5) -- (M57) -- (V6);
      \draw (f7) -- (M79) -- (V8);
    \end{scope}
  \end{tikzpicture}
  \caption{Flow chart}
  \label{fig:flow-chart}
\end{figure}

We generated random data of varying sample size.  For each fixed sample size, we generated 100 data samples.  The
simulation results are shown in Table~\ref{table5by10}.  For each simulated sample, we compute the sets $F_{1,V_{i}}$ and
$F_{2,V_{i}}$ as described above.  When computing $F_{1,V_{i}}$, we found that 2 iterations actually suffice.  Then we
checked whether $F_{2}$ is a proper subset of $I$ (second column), and we checked whether $F_{1}=F_{2}$ (third column).
Both for small and large sample sizes, we found that the $F_{1}=F_{2}$ quite often.
\begin{table} 
  \centering
  \caption{facial set approximation of $5 \times 10$ grid graph}
  \begin{tabular}{ddd}
    \toprule
    \multicolumn{1}{c}{sample size} & 
    \multicolumn{1}{c}{$F_2\neq I$} & \multicolumn{1}{c}{$F_1=F_2$} \\
    \midrule
      50 & 100.0\% & 94.3\% \\
     100 & 100.0\% & 82.5\% \\
     150 &  99.9\% & 76.5\% \\
     200 &  99.6\% & 81.2\% \\
     300 &  96.4\% & 87.7\% \\
     400 &  92.9\% & 91.5\% \\
     500 &  84.8\% & 93.9\% \\
    1000 &  44.7\% & 99.9\% \\
    \bottomrule
  \end{tabular}
  \label{table5by10}
\end{table}

We also investigated what happens when the outer approximation is not computed using all $3\times 5$-subgrids, but only a cover of four $3\times 5$-subgrids and one $2\times5$-subgrid (as in Figure~\ref{fig:local5}).
In all  simulations, this easier approximation gave the same result.  The same is not true for the inner
approximation: when using just one of the two families of parallel separators we obtain an inner approximation that is
much too small.

\section{Conclusion}
\label{sec:conclusion}

As mentioned before, previous work had made it  possible to identify $F_t$ for hierarchical models with up to 16 variables. In this paper, we offer a methodology to approximate, and sometimes, completely identify the facial set $F_t$ for high-dimensional models.  To find an inner and an outer approximation to $F_t$, first, we divided the original problem into subproblems of dimension at most 16 for which we could use linear programming  and, second, we combined the facial sets of the subproblems and related then to $F_t$. Identifying the subproblems and relating the facial sets to $F_t$ is numerically easy and the  corresponding software can be obtained upon request, from the authors.

It has long been established that determining the existence of the MLE is essential to correct statistical inference. In our paper, we have emphasized the problem of parameter estimation and shown how working with the likelihood $l_{F_2}$ yields much better estimates of the parameter that when working with $l$.
When testing one model versus another, the correct degrees of freedom for the asymptotic distribution of the test statistic is the difference between the dimensions of the facial sets for the two models being compared and not the difference between the dimensions of the two models.
If we only know approximations $F_1$ and $F_2$, we can use their dimensions to approximate the correct  degrees of freedom.  

In high dimensions, when the likelihood functions has so many terms that the (E)MLE cannot be computed,
a popular approach is to compute the maximum composite likelihood estimate. We have shown through an example
that,  when the global MLE does not exist, the local MLE for some of the same parameters might not exist either.  So, combining the values of the MLE of local likelihoods without being aware that the data lies on a face of the marginal polytope, one might also obtain misleading estimates of the parameters through composite likelihood.

We have not addressed the question of how to obtain reliable confidence intervals for the parameters by exploiting the properties of the inner and outer approximations to $F_t$. This subject clearly deserves attention and should be the subject of further work.

\appendix

\renewcommand{\theequation}{\thesection\arabic{equation}}

\section{Parametrizing hierarchical models}
\label{S-sec:notation-hier}

In this section, we recall the usual parametrization of hierarchical models, see, for example,
\cite{LetacMassam12:Bayes_regularization}.  The starting point is the parametrization~\eqref{H2prime-redundant} of the
hierarchical model, which we repeat here for convenience:
\begin{equation}
  \label{S-H2prime-redundant}
  \log p(i)=\sum_{D\in \Delta} \theta_D(i_D)
\end{equation}
This parametrization is not identifiable; that is, for any joint distribution~$p$ from the
hierarchical model there are different choices for the functions~$\theta_{D}$ that satisfy~\eqref{S-H2prime-redundant}.
One way to make the parameters unique is to choose a special element within each set $I_v$, which we denote by~0.  The
choice of 0 is arbitrary, and a different choice of 0 leads to a simple affine change of parameters.  With this choice,
the functions~$\theta_{D}$ become unique if one requires $\theta_{D}(i_{D})=0$ whenever $i_{v}=0$ for some~$v\in D$.

A parametrization in terms of real numbers is obtained using the following definitions: for $i\in I,$ we write
\begin{align*}
  S(i)&=\{v\in V\ ; \ i_{v}\neq 0\}, &
  J&=\{ j\in I\setminus\{0\},\ \ S(j)\in \Delta\}.
\end{align*}
For any $j\in J$, let
\begin{equation*}
  \theta_j = \theta_D(i_D)\text{ for the unique }i\in I\text{ with }S(i)=D,\;i_D=j_D.
\end{equation*}
To simplify the notation, we write $j\tl i$ whenever $S(j)\subseteq S(i)$ and $j_{S(j)}=i_{S(j)}$.
It is convenient to introduce the vectors
\begin{equation*}
  f_i= \sum_{j\in J:j\tl i}e_j,\;\;i\in I
\end{equation*}
where $e_j, j\in J$ are the unit vectors in $R^J$.  Moreover, let $A$ be the $J\times I$ matrix with columns~$f_{i}$,
$i\in I$, and let $\tilde A$ be the $(1+|J|)\times I$ matrix with columns equal to $\binom{1}{f_{i}}$, $i\in I$.
Then~\eqref{H2prime-redundant} can be rewritten in the following equivalent forms 

\begin{equation}
  \label{S-HXtilde}
  \log p_{\theta}(i) = \sum_{j\in J: j\tl i}\theta_j - k(\theta)
  = \langle \theta, f_i\rangle - k(\theta)
  = A^{t}\theta - k(\theta) = \tilde A^{t}\tilde\theta,
\end{equation}
where $\tilde{\theta}=(\theta_{0}, \theta)$ as a column vector and 
\begin{equation}
\label{S-all}
 -\theta_0= k(\theta) = \log\Big(\sum_{i\in I}\exp\Big(\sum_{j\in J: j\tl i} \theta_j\Big)\Big)
\end{equation}
acts as a normalization constant.
If $n=(n(i),i\in I)$ denotes the $I$-di\-men\-sional column vector of cell counts, then
\begin{equation}
  \label{S-marginal}
  \tilde{A}n=\left(\begin{array}{c}N\\t\end{array}\right)
  \quad\text{ and }\quad
  A n= t,
\end{equation}
where $N=\sum_{i\in I}n(i)$ is the total cell counts and $t$ is the column vector of sufficient statistic with components equal to the $j_{S(j)}$-marginal counts
$n(j_{S(j)})$, i.e. $t=(t_j, j\in J)$ where $t_j=n(j_{S(j)})=\sum_{i\mid i_{S(j)}=j_{S(j)}}n(i),\; j\in J$.

It follows from \eqref{S-marginal} that $\frac{t}{N}=\sum_{i\in I}\frac{n(i)}{N}f_i$.  Therefore, $t$ belongs to the
convex polytope with extreme points $f_i,i\in I$. This polytope is the \emph{marginal polytope} of the hierarchical
model, denoted by~$\Pbf_\Delta$.

\begin{example}
  \label{S-ex22} 
  Let $V=\{a,b,c\}$, $I_a=\{0,1\}=I_b=I_c$ and $\Delta=\{a,b,c,ab,bc\}$.
  Then
  \begin{gather*}
    I =(000,100,010,110,001,101,011,111), \\
    J=\{(100),(010),(001),(110),(011)\}, \\[16pt]
\tilde{A} = \begin{pmatrix}
\overmat{$f_{000}$}{1} & \overmat{$f_{001}$}{1} & \overmat{$f_{010}$}{1} & \overmat{$f_{011}$}{1} & \overmat{$f_{100}$}{1} & \overmat{$f_{101}$}{1} & \overmat{$f_{110}$}{\enskip1\enskip} & \overmat{$f_{111}$}{1} \\
0 & 1 & 0 & 1 & 0 & 1 & 0 & 1 \\
0 & 0 & 1 & 1 & 0 & 0 & 1 & 1 \\
0 & 0 & 0 & 0 & 1 & 1 & 1 & 1 \\
0 & 0 & 0 & 1 & 0 & 0 & 0 & 1 \\
0 & 0 & 0 & 0 & 0 & 0 & 1 & 1
\end{pmatrix}
\;
\begin{matrix}
\theta_{000} \\
\theta_{100} \\
\theta_{010} \\
\theta_{001} \\
\theta_{110} \\
\theta_{011}
\end{matrix}
\end{gather*}
\end{example}

\section{Example: Two binary random variables}
\label{S-sec:two-binaries}

Consider two binary random variables, and let $\Delta=\{\emptyset, \{1\}, \{2\}, \{1,2\}\}$.  The hierarchical model
$\Ecal_{\Delta}$ is the \emph{saturated model}; that is, it contains all possible probability distributions with full
support.  Then\nopagebreak
\vspace{4mm}
\begin{equation*}
    \tilde A = \begin{pmatrix}
      \overmat{$f_{00}$}{1} & \overmat{$f_{01}$}{1} & \overmat{$f_{10}$}{1} & \overmat{$f_{11}$}{1} \\
      0 & 1 & 0 & 1 \\
      0 & 0 & 1 & 1 \\
      0 & 0 & 0 & 1
    \end{pmatrix}
    \;
    \begin{matrix}
      \theta_{00} \\ \theta_{01} \\ \theta_{10} \\ \theta_{11}
    \end{matrix}
  \end{equation*}
  The marginal polytope is a 3-simplex (a tetrahedron) with facets
  \begin{gather*}
    \Fbf_{00}: 1 - t_{01} - t_{10} + t_{11} \ge 0, \quad
    \Fbf_{01}: t_{01} - t_{11} \ge 0, \\
    \Fbf_{10}: t_{10} - t_{11} \ge 0, \quad
    \Fbf_{11}: t_{11} \ge 0.
  \end{gather*}
  Each of the corresponding facets contains three columns of~$A$.  In fact, the facet $\Fbf_{i}$ in the above
  list does not contain the column $f_{i}$ of~$A$.

  The EMLE of the saturated model is just the empirical distribution; that is, $\gmle=\frac{1}{N}n$.  Suppose that $t$
  lies on the facet~$\Fbf_{00}$ (i.e.~$n=(0,n_{01},n_{10},n_{11})$ with $n(01), n(10), n(11)>0$).  If
  $p_{\theta^{(s)}}\to\gmle$, then $p_{\theta^{(s)}}(00)\to 0$, while all other probabilities converge to a non-zero
  value.  It follows that
  \begin{align*}
    \theta^{(s)}_{00}& =\log p_{\theta^{(s)}}(00)\to-\infty, \\
    \theta^{(s)}_{01}& =\log \frac{p_{\theta^{(s)}}(01)}{p_{\theta^{(s)}}(00)} \to +\infty, \\
    \theta^{(s)}_{10}& =\log \frac{p_{\theta^{(s)}}(10)}{p_{\theta^{(s)}}(00)} \to +\infty, \\
    \theta^{(s)}_{11}& =\log \frac{p_{\theta^{(s)}}(11)p_{\theta^{(s)}}(00)}{p_{\theta^{(s)}}(01)p_{\theta^{(s)}}(10)} \to -\infty.
  \end{align*}
  On the other hand, $\theta^{(s)}_{01}+\theta^{(s)}_{00}=\log p_{\theta^{(s)}}(01)$ converges to a finite value, as do
  $\theta^{(s)}_{10}+\theta^{(s)}_{00}=\log p_{\theta^{(s)}}(10)$ and $\theta^{(s)}_{11}+\theta^{(s)}_{01}=\log
  p_{\theta^{(s)}}(11)/p_{\theta^{(s)}}(10)$.

  Proceeding similarly for the other facets, one can show for the limits
  $\theta_{ij}:=\lim_{s\to\infty}\theta_{ij}^{(s)}$:
  \begin{center}
    \begin{tabular}{lccccc}
      \toprule
      & $\theta_{00}$ & $\theta_{01}$ & $\theta_{10}$ & $\theta_{11}$ & finite parameter combinations:\\
      \midrule
      $\Fbf_{00}$ & $-\infty$ &  $+\infty$ &  $+\infty$ &  $-\infty$ & $\theta^{(s)}_{01}+\theta^{(s)}_{00}$, $\theta^{(s)}_{10}+\theta^{(s)}_{00}$, $\theta^{(s)}_{11}+\theta^{(s)}_{01}$ \\
      $\Fbf_{01}$ &   finite  &  $-\infty$ &    finite  &  $+\infty$ & $\theta^{(s)}_{00}$, $\theta^{(s)}_{10}$, $\theta^{(s)}_{01}+\theta^{(s)}_{11}$ \\
      $\Fbf_{10}$ &   finite  &    finite  &  $-\infty$ &  $+\infty$ & $\theta^{(s)}_{00}$, $\theta^{(s)}_{01}$, $\theta^{(s)}_{10}+\theta^{(s)}_{11}$ \\
      $\Fbf_{11}$ &   finite  &    finite  &    finite  &  $-\infty$ &  $\theta^{(s)}_{00}$, $\theta^{(s)}_{10}$, $\theta^{(s)}_{01}$ \\
      \bottomrule
    \end{tabular}
  \end{center}
  Each line of the last column contains three combinations of the parameters $\theta^{(s)}_{i}$ that converge to a finite value.
 Any other parameter combination that converges is a linear
  combination of these three.  This can be seen by using the coordinates $\mu_{i}$ introduced in
  Section~\ref{sec:param-trafos}.  For example, on the facet~$\Fbf_{01}$,
  consider the parameters
  \begin{gather*}
    \mu_{10} = \log p(10)/p(00) = \theta_{10}, \qquad
    \mu_{11} = \log p(11)/p(00) = \theta_{10}+\theta_{01}+\theta_{11}, \\
    \mu_{01} = \log p(01)/p(00) = \theta_{01}.
  \end{gather*}
  Then $\mu_{10}$ and~$\mu_{11}$ are identifiable parameters on~$\Ecal_{F_{01}}$, and $\mu_{01}$ diverges close
  to~$\Fbf_{01}$.  By Lemma~\ref{lem:mu-properties}, the linear combinations that are well-defined are $\mu_{10}=\theta_{10}$
  and~$\mu_{11}=\theta_{10}+(\theta_{01}+\theta_{11})$.  The above table also lists~$\theta_{00}$, which  is not a linear combination of those but that is fine because it is not  free.

We obtain similar results for the facets $\Fbf_{01}$ and $\Fbf_{11}$. The results are summarized in the following table:
\begin{center} 
  \begin{tabular}{cccc}
    \hline
    facet & $\mu_{01}$ & $\mu_{10}$ & $\mu_{11}$ \\
    \hline
    $\mathbf{F}_{01}$ &  $-\infty$ & finite & finite \\
    $\mathbf{F}_{10}$ &  finite & $-\infty$ & finite \\
    $\mathbf{F}_{11}$ &  finite & finite & $-\infty$ \\
    \hline
  \end{tabular}
\end{center}
Of course, by definition of the $\mu_i$s, we cannot consider the facet $\mathbf{F}_{00}$ where $n(00)=0$.  To study
$\mathbf{F}_{00}$, we have to choose another zero cell and redefine the parameters~$\mu_{i}$.

The situation is more complicated for faces smaller than facets, because sending a single parameter to plus
or minus infinity can be enough to send the distribution to a face $F$ of higher codimension, as we will see below.  The
remaining parameters then determine the position within~$\Ecal_{\Delta,F}$.  Thus, in this case there are more remaining
parameters than the dimension of~$\Ecal_{\Delta,F}$.

  For example, the data vector $n=(n_{00},0,n_{10},0)$ (with $n_{00},n_{10}>0$) lies on the face
  $\Fbf=\Fbf_{01}\cap\Fbf_{11}$ of codimension two.  If $p_{\theta^{(s)}}\to\gmle$, then
  \begin{align*}
    \theta^{(s)}_{00} &= \log p_{\theta^{(s)}}(00) \to \log\frac{n_{00}}{N}, \\
    \theta^{(s)}_{01} &= \log \frac{p_{\theta^{(s)}}(01)}{p_{\theta^{(s)}}(00)} \to -\infty, \\
    \theta^{(s)}_{10} &= \log \frac{p_{\theta^{(s)}}(10)}{p_{\theta^{(s)}}(00)} \to \log\frac{n_{10}}{n_{00}}.
  \end{align*}
  However, the limit of
  $\theta^{(s)}_{11}=\log\frac{p_{\theta^{(s)}}(11)p_{\theta^{(s)}}(00)}{p_{\theta^{(s)}}(01)p_{\theta^{(s)}}(10)}$ is
  not determined.  The only constraint is that $\theta^{(s)}_{11}$ cannot go to $+\infty$ faster than
  $\theta^{(s)}_{01}$ goes to~$-\infty$, since
  $p_{\theta^{(s)}_{11}}=\exp(\theta^{(s)}_{00}+\theta^{(s)}_{01}+\theta^{(s)}_{10}+\theta^{(s)}_{11})$ has to converge
  to zero.

  With the same data vector $n=(n_{00},0,n_{10},0)$, suppose we use a numerical algorithm to optimize the likelihood
  function by optimizing the parameters~$\theta_{j}$ in turn.  To be precise, we order the parameters $\theta_{j}$ in
  some way.  For simplicity, say that the parameters are $\theta_{1},\theta_{2},\dots,\theta_{h}$.  Then we let
  \begin{equation*}
    \theta_{j}^{(k+1)} = \arg\max\limits_{y\in\R} l(\theta_{1}^{(k+1)},\dots,\theta_{j-1}^{(k+1)},y,\theta_{j+1}^{(k)},\dots,\theta_{h}^{(k)})
  \end{equation*}
  (this is called the \emph{non-linear Gauss-Seidel method}).  Let us choose the ordering
  $\theta_{01},\theta_{10},\theta_{11}$ (note that $\theta_{00}=-k(\theta)$ is not a free parameter).  We start at
  $\theta_{01}^{(0)}=\theta_{10}^{(0)}=\theta_{11}^{(0)}=0$.  In the first step, we only look at~$\theta_{01}$.  That
  is, we want to solve
  \begin{multline}
    \label{S-eq:ex:critical-eq}
    0 = \frac{\partial}{\partial\theta_{01}} l(\theta) = - \frac{\exp(\theta^{(1)}_{01}) + \exp(\theta^{(1)}_{01}+\theta^{(0)}_{10}+\theta^{(0)}_{11})}{1 + \exp(\theta^{(1)}_{01}) + \exp(\theta^{(0)}_{10}) + \exp(\theta^{(1)}_{01}+\theta^{(0)}_{10}+\theta^{(0)}_{11})}
    \\
    = - \frac{2\exp(\theta^{(1)}_{01})}{1 + 2\exp(\theta^{(1)}_{01})}.
  \end{multline}
  This derivative is negative for any finite value of~$\theta^{(1)}_{01}$, and thus the critical equation has no
  finite solution.  If we try to solve this equation numerically, we will find that $\theta^{(1)}_{01}$ will be a large
  negative number.  Next, we look at $\theta_{10}$.  We fix the other variables and try to solve
  \begin{multline*}
    0 = \frac{\partial}{\partial\theta_{10}} l(\theta) = \frac{n_{10}}{N} - \frac{\exp(\theta^{(1)}_{10}) + \exp(\theta^{(1)}_{01}+\theta^{(1)}_{10}+\theta^{(0)}_{11})}{1 + \exp(\theta^{(1)}_{01}) + \exp(\theta^{(1)}_{10}) + \exp(\theta^{(1)}_{01}+\theta^{(1)}_{10}+\theta^{(0)}_{11})}
    \\
    \approx \frac{n_{10}}{N} - \frac{\exp(\theta^{(1)}_{10})}{1 + \exp(\theta^{(1)}_{10})},
  \end{multline*}
  where we have used that $\theta^{(1)}_{01}$ is a large negative number.  This equation always has a unique
  solution
  \begin{equation*}
    \theta^{(1)}_{10}\approx \log\frac{n_{10}}{N-n_{10}}.
  \end{equation*}
  Finally, we look at $\theta_{11}$.  We have to solve
  \begin{equation*}
    0 = \frac{\partial}{\partial\theta_{11}} l(\theta) = - \frac{\exp(\theta^{(1)}_{01}+\theta^{(1)}_{10}+\theta^{(1)}_{11})}{1 + \exp(\theta^{(1)}_{01}) + \exp(\theta^{(1)}_{10}) + \exp(\theta^{(1)}_{01}+\theta^{(1)}_{10}+\theta^{(1)}_{11})} \approx 0.
  \end{equation*}
  Actually, this equation again has no solution,
  and the numerical solution for $\theta^{(1)}_{11}$ should be close to numerical minus infinity.  However, since
  $\theta^{(1)}_{01}$ is already close to $-\infty$, the equation is already approximately satisfied.  Thus, there is no
  need to change $\theta_{11}$.  In simulations, we observed that usually $\theta^{(1)}_{11}$ will be negative, but not
  as negative as~$\theta^{(1)}_{01}$.  In theory, we would have to iterate and now optimize~$\theta_{01}$ again.  But
  the values will not change much, since the critical equations are already satisfied to a high numerical precision
  after one iteration.

  It is not difficult to see that the result is different if we change the order of the variables.  If $\theta_{11}$ is
  optimized before~$\theta_{01}$, then $\theta_{11}^{1}$ will in any case be a large negative number.

  For general data, the derivative of $l(\theta)$ with respect to~$\theta_{01}$ (equation~\eqref{S-eq:ex:critical-eq}) takes the form
  \begin{equation*}
    \frac{\partial}{\partial\theta_{01}} l(\theta) = \frac{t_{01}}{N} - \frac{\exp(\theta^{(1)}_{01}) + \exp(\theta^{(1)}_{01}+\theta^{(0)}_{10}+\theta^{(0)}_{11})}{1 + \exp(\theta^{(1)}_{01}) + \exp(\theta^{(0)}_{10}) + \exp(\theta^{(1)}_{01}+\theta^{(0)}_{10}+\theta^{(0)}_{11})}.
  \end{equation*}
  Setting this derivative to zero and solving for~$\theta^{(1)}_{01}$ leads to a linear equation in~$\theta^{(1)}_{01}$
  with symbolic solution
  \begin{equation*}
    \theta^{(1)}_{01} = \log \frac{1+\exp(\theta^{(0)}_{10})}{1 + \exp(\theta^{(0)}_{10}+\theta^{(0)}_{11})}\frac{\frac{t_{01}}{N}}{1 - \frac{t_{01}}{N}}.
  \end{equation*}
  In fact, for any hierarchical model, the likelihood equation is linear in any single parameter~$\theta_{j}$, as long
  as all other parameters are kept fixed (more generally this is true when the design matrix~$A$ is a 0-1-matrix).
  Instead of optimizing the likelihood numerically with respect to one parameter, it is possible to use these symbolic
  solutions.  This leads to the Iterative Proportional Fitting Procedure (IPFP).  In our example, the IPFP would lead to
  a division by zero right in the first step, indicating that the MLE does not exist
  (unfortunately, IPFP does not always fail that quickly when the MLE does not exist).

\section{Parametrizations adapted to facial sets}
\label{S-sec:best-parameters}

Let us briefly discuss how to remedy problems 1.\ (identifiability), 2.\ (relation between parameters on $\Ecal$ and $\Ecal_{F_{t}}$) and 3.\ (cancellation of infinities in linear combinations of diverging parameters) from the beginning of Section~\ref{sec:param-trafos}.  The idea
to remedy 1.\ and 2.\ is to define parameters $\mu_{i}$, $i\in L$, of~$\Ecal_{A}$ such that a subset $L_{t}\subseteq L$
of the parameters parametrizes~$\Ecal_{F_{t},A}$ in a consistent way.  Denote by $A^{\mu}=(a^{\mu}_{j,i}, j\in L, i\in
I)$ the design matrix of $\Ecal_{A}$ corresponding to the new parameters~$\mu$.  Then the necessary conditions are:
\begin{enumerate}
\item[($*$)] Let $A^{\mu}_{L_{t},F_{t}}:=(a^{\mu}_{j,i}, j\in L_{t}, i\in F_{t})$ be the submatrix of~$A^{\mu}$ with
  rows indexed by~$L_{t}$ and columns indexed by~$L_{t}$, and denote by $\tilde A^{\mu}_{L_{t},F_{t}}$ the same matrix
  with an additional row of ones.  The rank of $\tilde A^{\mu}_{L_{t},F_{t}}$ is equal to~$|L_{t}|+1$, the number of
  its rows (and thus, $A^{\mu}_{L_{t},F_{t}}$ has rank~$|L_{t}|$).
\item[($**$)] $a^{\mu}_{j,i}=0$ for all $i\in F_{t}$ and $j\in L\setminus L_{t}$.
\end{enumerate}
In fact, ($**$) implies that $A^{\mu}_{L_{t},F_{t}}$ is the design matrix of~$\Ecal_{A,F_{t}}$, since the parameters
$\mu_{i}$ with $i\notin L_{t}$ do not play a role in the parametrization $\mu\mapsto p_{F_{t},\mu}$.  Moreover, ($*$)
implies that the parametrization $\mu\mapsto p_{F_{t},\mu}$ is identifiable.  In this sense, we have remedied
problem~1.

Since $\tilde A^{\mu}_{L_{t},F_{t}}$ has full row rank, it has a right inverse matrix~$\tilde C$, such that
$\tilde A^{\mu}_{L_{t},F_{t}}\tilde C=I_{|L_{t}|+1}$ equals the identity matrix of size~$|L_{t}|+1$.  Recall that
\begin{align*}
  \log p_{F_{t},\mu}(i) &= \<\mu^{t}, f^{\mu}_{i} \> - k_{F}(\mu), \\
  \log p_{\mu}(i) &= \<\tilde\mu^{t}, f^{\mu}_{i} \> - k(\mu),
\end{align*}
for any parameter vector~$\mu$ and all $i\in F_{t}$.  Since $f^{\mu}_{i}$ are the columns of~$A^{\mu}$ and since the
components of $f^{\mu}_{i}$ corresponding to $L\setminus L_{t}$ vanish by ($**$), we may apply the matrix~$C$ obtained
from $\tilde C$ by dropping the row corresponding to $k_{F}$ or $k$ and obtain
\begin{equation}
  \label{S-eq:mu-inverse}
  (\log p_{\mu})C = \mu_{L_{t}}
  \quad\text{ and }\quad
  (\log p_{F_{t},\mu})C = \mu_{L}.
\end{equation}
When $p_{\mu^{(s)}}$ is a sequence in $\Ecal_{A}$ with limit $p_{\mu}$ in~$\Ecal_{F_{t},A}$,
then~\eqref{S-eq:mu-inverse} shows that $\mu_{i}^{(s)}\to\mu_{i}$ for~$i\in L_{t}$.  In this sense, we have remedied
problem~2.

Finally, we solve problem~3.  Suppose that we have chosen parameters $\mu_{L}$ as in Section~\ref{sec:param-trafos}, and
let $A^{\mu_{L}}$ be the design matrix with respect to these parameters.  Then $(A^{\mu_{L}})_{j,i}=0$ if $i\in F_{t}$
and $j\notin L_{t}$.  Moreover, for $j\in L_{t}$, the $j$th column 
of~$A_{\mu_{L}}$ has a single non-vanishing entry (equal to one) at position~$j$.  Suppose that $F_{t}$ corresponds to a
face $\Fbf_{t}$ of codimension~$c$.  Then there are $c$ facets of $\Pbf$ whose intersection is~$\Fbf_{t}$.  Thus,
following the notation introduced in Remark~\ref{homog}, there exist $c$ inequalities
\begin{equation}
  \label{S-eq:param-trafo-inequalities}
  \<\tilde g_{1}, \tilde x\> \ge 0,\quad\dots,\quad \<\tilde g_{c}, \tilde x\> \ge 0
\end{equation}
that together define~$\Fbf_{t}$.  In this case, the vectors $\tilde g_{1},\dots,\tilde g_{c}$ are linearly independent
and satisfy $\<\tilde g_{j},\tilde f_{i}\>=0$ (thus, they are a basis of the kernel of~$(\tilde
A^{\mu_{L}}_{F_{t}})^{t}$).  It follows that the $k$th component of $g_{j}$, denoted by~$g_{j,k}$, vanishes if $k\in
L_{t}$; that is, the inequalities~\eqref{S-eq:param-trafo-inequalities} do not involve the variables corresponding
to~$L_{t}$.  Let $G$ be the square matrix, indexed by~$L\setminus L_{t}$ with entries $g_{j,k}$, $j,k\in L\setminus
L_{t}$.  Then the square matrix
\begin{equation*}
  \tilde G =
  \begin{pmatrix}
    1 & 0 \\
    0 & G
  \end{pmatrix}
\end{equation*}
is invertible.  We claim that the parameters $\lambda = \tilde G^{-1}\mu_{L}$ are what we are looking for.

The design matrix with respect to the parameters~$\lambda$ is $A^{\lambda}=\tilde G A^{\mu_{L}}$.  For any $j\notin
L_{t}$,
\begin{align*}
  A^{\lambda}_{j,i} & = 0, \quad \text{ if }i\in F_{t}, & \text{and }\quad 
  A^{\lambda}_{j,i} & = \<\tilde g_{j},\tilde f_{i}\> \ge 0, \quad \text{ if }i\notin F_{t}.
\end{align*}
This implies the following properties:
\begin{enumerate}
\item If all parameters $\lambda_{j}$ with $j\notin L_{t}$ are sent to $-\infty$, then $p_{\lambda}$ tends towards a
  limit distribution with support~$F_{t}$.
\item The coefficient of $\lambda_{j}$ in any log-probability is non-negative, so there is no cancellation of
  $\pm\infty$.
\end{enumerate}

So far, we only used the fact that the vectors $\tilde g_{j}$ define valid inequalities for the face~$\Fbf_{t}$.
Suppose that we choose $\tilde g_{j}$ in such a way that each inequality $\<\tilde g_{j},\tilde x\>\ge 0$ defines a
facet.  The intersection of less than $c$ facets is a face that strictly contains~$\Fbf_{t}$.  This implies that for
each~$j$, there exists $i_{j}\in I\setminus F_{t}$ such that $f_{i_{j}}$ satisfies
\begin{align*}
  \<\tilde g_{j}, \tilde f_{i_{j}}\> & > 0, & \text{ and }\quad
  \<\tilde g_{j'}, \tilde f_{i_{j}}\> & = 0 \text{ for all }j'\neq j,
\end{align*}
and so
\begin{align*}
  A^{\lambda}_{j,i_{j}} & > 0, & \text{ and }\quad
  A^{\lambda}_{j'},i_{j} & = 0 \text{ for all }j'\neq j.
\end{align*}
Hence:
\begin{enumerate}[resume]
\item If $\lambda^{(s)}_{j}$ are sequences of parameters such that $p_{\lambda^{(s)}}$ tends towards a limit
  distribution with support~$F_{t}$, then $\lambda^{(s)}_{j}\to-\infty$ for all $j\notin L_{t}$.
\end{enumerate}

It is not difficult to see that, conversely, any parametrization that satisfies these three properties comes from facets
defining the face~$\Fbf_{t}$.

\section{Uniform sampling for the \texorpdfstring{$4\times4$}{4x4}-grid}
\label{S-sec:uniform-4x4}
This section enhances the example in Section~\ref{sec:4x4}.
In a second experiment, we generated random samples from the uniform distribution, that is from the probability distribution $P_{\theta}$
in the hierarchical model where all parameters $\theta_{j}$, $j\in J$, are set to zero. 
For each sample size, \num{1000} samples were obtained.  The results are given in the following table:
\begin{center}
  \begin{tabular}{dddd}
    \toprule
    \multicolumn{1}{c}{sample size} & \multicolumn{1}{c}{MLE does not exist} & \multicolumn{1}{c}{$F_1=F_t$} & \multicolumn{1}{c}{$F_2=F_{t}$} \\  
    \midrule
    10 & 98.5\% &  96.3\% & 100.0\% \\
    15 & 68.9\% &  99.9\% & 100.0\% \\
    20 & 29.0\% & 100.0\% & 100.0\% \\
    50 &  0.0\% & 100.0\% & 100.0\% \\
    \bottomrule
  \end{tabular}
  \label{S-tab:4by4_uniform}
\end{center}
As the table shows, for larger
samples the probability that a random sample lies on a proper face becomes very small.  If $F_{t}=I$, then clearly
$F_{t}=F_{2}$.  But we also found $F_{t}=F_{2}$ for all samples with $t$ lying on a proper face, which shows that
$F_{2}$ is an excellent approximation of $F_{t}$ in this model.  For the inner approximation, we observed some samples
with $F_{1}\neq F_{t}$, but they seem to be very rare.

\section{Estimated cell frequencies for the NLTCS data}
\label{S-sec:NLTCS-frequencies}

The following table lists the estimates of the top five cell counts obtained using our method and compares them with those obtained by other methods in \citet{DobraLenkoski11:Copula_Gaussian_graphical_models}.
\begin{center}
  \small
  \begin{tabular}{cccccc}
    \toprule
    Support of Cell & Observed & GoM & LC &CGGMs & MLE on facial set \\
    \midrule
    $\emptyset$ & 3853 & 3269 & 3836.01 & 3767.76 & 3647.4 \\
    $\{10\}$ & 1107 & 1010 & 1111.51 & 1145.86 & 1046.9 \\
    $\{1:16\}$ & 660 & 612 & 646.39 & 574.76 &  604.4 \\
    $\{5\}$ &  351 & 331 & 360.52 & 452.75 &  336 \\
    $\{5,10\}$ & 303 & 273 & 285.27 & 350.24 &  257.59 \\
    $\{12\}$ & 216 & 202 & 220.47  & 202.12 &  239.24 \\
    \bottomrule
  \end{tabular}
\end{center}



\section{A Linear Programming algorithm to compute facial sets}
\label{S-linear programming}
 
Let $A$ be the design matrix, let $A_{+}$ be the sub-matrix with columns indexed by the positive cells $I_{+}$, and let $A_{0}$ be the sub-matrix indexed by the empty cells (thus $A=(A_{+},A_{0})$ after reordering the columns).
\begin{lemma}
  Any solution $g^*$ of the non-linear problem
  \begin{equation}
    \label{S-lm:face}
    \max \; \lVert Ag \rVert_0
    \qquad\text{subject to }
    \begin{cases}
      A_{+} g=0, \\ 
      A_0 g \geqslant 0,
    \end{cases}
  \end{equation}
  is perpendicular to the smallest face containing~$t$. The corresponding facial set is  $F_t=I \setminus \supp(Ag^*)$.
\end{lemma}

The optimization problem \eqref{S-lm:face} is highly non-linear and non-convex.  It can be solved by repeatedly solving the associated $\ell_1$-norm optimization problem:
\begin{equation}
   \label{S-linearface}
   \max \; \lVert A_0g \rVert_1 
   \qquad\text{subject to }
   \begin{cases}
     A_{+} g=0, \\ 
     A_0 g \ge 0, \\ 
     A_0g \le 1.
   \end{cases}
 \end{equation}
 Due to the constraints, $\lVert A_0g \rVert_1 = A_{0}g$, and so Problem~\eqref{S-linearface} is a linear programming
 problem.  It can be solved repeatedly until the smallest facial set $F_t$ is obtained. The process is as follows:
\begin{algorithm}
\caption{Computing $F_{t}$ Linear Programming}
\begin{algorithmic}
\REQUIRE Design matrix $A$ and positive cell index $I_{+}$
\STATE Let $A_{+}:=A(I_{+},:)$, $A_{0}:=A\setminus A_{+}$
\REPEAT{}
\STATE Solve problem \eqref{S-linearface}, get the solution $g^*$ and the corresponding maximum $z^*$
\STATE $A_{0} :=$  the submatrix of $A_0$ with those columns $f_{i}$ of $A_0$ that satisfy $\<f_i,g^*\> =0$.
\UNTIL {$A_0 = \emptyset$ or $z^{*} \not = 0$ }
\IF {$A_0=\emptyset$}
\STATE $F_t:=I_{+}$
\ENDIF
\IF {$z^*=0$}
\STATE $F_t:=I_{+}\cup \big\{i \;\big|\; \text{column }i\text{ of $A$ belongs to }A_0\big\}$
\ENDIF
\end{algorithmic}
\end{algorithm}

The algorithm is introduced in the supplementary material to~\citep{FienbergRinaldo12:MLE_in_loglinear_models}, where it
is also proved that it outputs the correct result.

\section{Some proofs}
\label{sec:S-proofs}

\subsection{Proof of Theorem~\ref{thm:closure-of-expfam}}
\label{S-sec:proof1}

Theorem~\ref{thm:closure-of-expfam} goes back to~\cite{Barndorff:exponential_families}, who studies the closure of much more general exponential
families.  The case of a discrete exponential family is much easier.

The theorem follows from the following lemmas:
\begin{lemma}
  \label{S-lem:p-is-truncation}
  Let $p\in\ol{\Ecal_{A}}$.  Then $p\in\Ecal_{A,\supp(p)}$.
\end{lemma}
\begin{lemma}
  \label{S-lem:any-truncation}
  Let $p\in\ol{\Ecal_{A}}$.  Then $\Ecal_{A,\supp(p)}\subseteq\ol{\Ecal_A}$.
\end{lemma}
\begin{lemma}
  \label{S-lem:supp-is-facial}
  Let $p\in\ol{\Ecal_{A}}$.  Then $\supp(p)$ is facial.  
\end{lemma}
\begin{lemma}
  \label{S-lem:facial-is-supp}
  If $F$ is facial, then there exists $p\in\ol{\Ecal_{A}}$ with~$\supp(p)=F$.
\end{lemma}
Indeed, Lemma~\ref{S-lem:p-is-truncation} shows that $\ol{\Ecal_A}\subseteq\bigcup_{F}\Ecal_{A,F}$, where the union is
over all support sets~$F$.  Lemma~\ref{S-lem:any-truncation} shows the converse containment, so that
$\ol{\Ecal_A}=\bigcup_{F}\Ecal_{A,F}$.  It remains to see that a subset $F\subseteq I$ is a support set if and only if
$F$ is facial.  This follows from Lemmas~\ref{S-lem:supp-is-facial} and~\ref{S-lem:facial-is-supp}.

In the proofs of Lemmas~\ref{S-lem:p-is-truncation} to \ref{S-lem:facial-is-supp}, we need the following easy lemma of which
we omit the proof:
\begin{lemma}
  \label{S-lem:implicit-ker}
  $p\in\Ecal_{A}$ if and only if $\log(p)\perp\ker A$.
\end{lemma}
\begin{proof}[Proof of Lemma~\ref{S-lem:p-is-truncation}]
  Let $p=\lim_{k\to\infty}p_{k}$, where $p_{k}\in\Ecal_{A}$, and let $F=\supp(p)$.  Then $\Ecal_{A,F}$ is the
  exponential family~$\Ecal_{A_{F}}$, where $A_{F}$ consists of the columns of~$A$ indexed by~$F$.  Any $v\in\ker A_{F}$
  can be extended by zeros to $v'\in\ker A$.  By Lemma~\ref{S-lem:implicit-ker},
  \begin{equation*}
    0 = \langle\log(p_{k}),v'\rangle = \sum_{i\in F}\log(p_{k}(i))v(i) \to \langle\log(p),v\rangle.
  \end{equation*}
  Thus, $\log(p)\perp\ker A_{F}$, which implies~$p\in\Ecal_{A,F}$.
\end{proof}
\begin{proof}[Proof of Lemma~\ref{S-lem:any-truncation}]
  Let $p=\lim_{k\to\infty}p_{k}$, where $p_{k}\in\Ecal_{A}$, let $F=\supp(p)$, and let~$q\in\Ecal_{A,F}$.  Then there
  exists parameters $\theta$ with $\log(q(i))-\log(p(i))=\langle \theta, f_{i}\rangle$ for all~$i\in F$.  For any~$k$,
  there exists $c_{k}>0$ such that $q_{k}:= c_{k}p_{k}\exp(\langle \theta, A\rangle)\in\Ecal_{A}$.
  Then $q_{k}\to q$ as $k\to\infty$, and so $q\in\ol{\Ecal_A}$.
\end{proof}
\begin{proof}[Proof of Lemma~\ref{S-lem:supp-is-facial}]
  Let $p=\lim_{k\to\infty}p_{k}$, where $p_{k}\in\Ecal_{A}$, and let $F=\face_{A}(\supp(p))$.
  Then $x=\frac{1}{|\supp(p)|}\sum_{i\in\supp(p)}f_{i}$ is an interior point of the face corresponding to~$F$, and thus
  there exist positive coefficients~$\lambda_{i}>0$, $i\in F$, with $x=\sum_{i\in F}\lambda_{i}f_{i}$.  The vector
  $v=(v_{i}, i\in I)$ defined by
  \begin{equation*}
    v_{i}=
    \begin{cases}
      \frac{1}{|\supp(p)|}-\lambda_{i}, & i\in\supp(p), \\
      -\lambda_{i}, & i\in F\setminus\supp(p), \\
      0, & i\notin F,
    \end{cases}
  \end{equation*}
  satisfies $Av=x-x=0$.
  By Lemma~\ref{S-lem:implicit-ker}, $\log(p_{k})\perp v$ for all~$k$.  In particular,
  \begin{equation*}
    \sum_{i\in F\setminus\supp(p)}\lambda_{i}\log(p_{k}(i)) = \sum_{i\in \supp(p)}\log(p_{k}(i))v_{i}
    \to \sum_{i\in \supp(p)}\log(p(i))v_{i}.
  \end{equation*}
  On the other hand, note that each coefficient $\lambda_{i}$ for $i\in F\setminus\supp(p)$ on the left hand side is
  positive, while $\log(p_{k}(i))\to-\infty$ for $i\notin\supp(p)$.  This shows that $F\setminus\supp(p)=\emptyset$.
\end{proof}
\begin{proof}[Proof of Lemma~\ref{S-lem:facial-is-supp}]
  If $F$ is facial, there exist $g\in\R^{h}$ and $c\in\R$ with $\langle g,f_{i}\rangle\ge c$ for all $i\in I$ and
  $\langle g,f_{i}\rangle=c$ if and only if~$i\in F$.  Let $\theta^{(s)}=-s\cdot g$.  Then
  \begin{equation*}
    k_{F}(\theta_{(s)}) + sc = \log \sum_{i\in I}\exp(- s\langle g,f_{i}\rangle + sc) 
    \to \log|F|,
  \end{equation*}
  and so
  \begin{multline*}
    \log p_{\theta^{(s)}}(i) = - s \langle g, f_{i}\rangle - k_{F}(\theta_{(s)}) 
    = (sc - s \langle g, f_{i}\rangle) - (k_{F}(\theta_{(s)}) + sc)
    \\
    \to
    \begin{cases}
      - \log|F|, & \text{ if } i\in F, \\
      -  \infty, & \text{ if } i\notin F,
    \end{cases}
  \end{multline*}
  as $s\to\infty$.  Thus, $p_{\theta^{(s)}}$ converges to the uniform distribution on~$F$.
\end{proof}

\subsection{Proof of Theorem~\ref{thm:gmle}}
\label{S-sec:proof2}

By definition, any EMLE $\gmle$ belongs to the closure of the model.  According to Theorem~\ref{thm:closure-of-expfam},
the support of $\gmle$ is facial.  If $\supp(p)$ does not contain~$\supp(n)$, then the log-likelihood goes to minus infinity, $\tilde
l(p)=-\infty$, and so~$p$ does not maximize the likelihood, Therefore, $\supp(\gmle)$ is a facial set containing~$\supp(n)$.
Thus, $F_{t}\subseteq\supp(\gmle)$.

By Lemma~\ref{S-lem:p-is-truncation}, $\gmle$ belongs to $\Ecal_{\Delta,\supp(\gmle)}$, which is parametrized by a
vector~$\theta$, see Theorem~\ref{thm:closure-of-expfam}.
On $\Ecal_{\Delta,\supp(\gmle)}$, the log-likelihood function in terms of this parameter $\theta$ is
\begin{equation*}
  l_{F}(\theta) = \sum_{j\in J}\theta_jt_j-N k_{F}(\theta).
\end{equation*}
The critical equations satisfied by $p_{*}$ are
\begin{equation*}
  A \gmle = \frac{t}{N},
\end{equation*}
proving the first property.  Note that these equations are independent of the parameters and the support of~$\gmle$.  We
now show that any solution to these equations is supported on the same face of~$\Pbf$ as~$\frac{t}{N}$.

Let $p$ be a probability distribution on~$I$ such that $\supp(p)$ does not contain~$F_t$.  This means that there is a linear inequality $\langle g,t\rangle\ge c$
that is valid on~$\Pbf$ and such that
(i) $\langle g,f_{i}\rangle=c$ for all~$i\in F_{t}$; and
(ii) $\langle g,f_{i}\rangle>c$ for some~$i\in\supp(p)$.
Then
\begin{equation*}
  \langle g, A p\rangle = \sum_{i}\langle g,f_{i}\rangle p(i) > c
  = \frac{1}{N}\sum_{i}n(i) \langle g, f_{i}\rangle
  = \langle g, \frac{t}{N}\rangle,
\end{equation*}
which implies $Ap\neq \frac{t}{N}$.  This shows $\supp(\gmle)\subseteq F_t$ and finishes the proof of $\supp(\gmle)=F_{t}$.

We have shown the two properties, and it remains to argue that the EMLE is unique.  But this follows from the fact that
$\supp(\gmle)$ is equal to~$F_{t}$, and $l_{F}$ is strictly convex (up to identifiability), such that there is a unique
distribution on~$\Ecal_{\Delta,F_{t}}$ that maximizes the likelihood.

\bibliographystyle{plainnat}
\bibliography{main}

\end{document}